\documentclass[reqno]{amsart}
\usepackage{hyperref}
\usepackage{amssymb}
\usepackage{color}
\definecolor{dkgreen}{rgb}{0,0.6,0}
\definecolor{gray}{rgb}{0.5,0.5,0.5}
\definecolor{mauve}{rgb}{0.58,0,0.82}
\numberwithin{equation}{section}
\newtheorem{theorem}{Theorem}[section]
\newtheorem{lemma}[theorem]{Lemma}
\newtheorem{proposition}[theorem]{Proposition}
\newtheorem{definition}[theorem]{Definition}

\newtheorem{examples}[theorem]{Examples}
\newtheorem{remark}[theorem]{Remark}
\newtheorem{corollary}[theorem]{Corollary}

\def\R{\mathbb R}

\begin{document}
\title
[\hfil Stepanov pseudo almost periodic functions and applications]{Stepanov pseudo almost periodic functions and applications}
\author[K. Khalil, M. Kosti\'c  and M. Pinto]
{K. Khalil, M. Kosti\'c and M. Pinto}
  
\address{K. Khalil\newline
 Department of Mathematics, Faculty of Sciences Semlalia, Cadi Ayyad University\\
 B.P. 2390, 40000 Marrakesh, Morocco.}
\email[Corresponding author]{kamal.khalil.00@gmail.com}

\address{M. Kosti\'c \newline
 Faculty of Technical Sciences, University of Novi Sad, Trg D. Obradovica 6, 21125 ´
Novi Sad, Serbia.}
\email{marco.s@verat.net}

\address{M. Pinto\newline
Departamento de
Matem\'aticas, Facultad de Ciencias, Universidad de Chile, Santiago de Chile, Chile.}
\email{pintoj.uchile@gmail.com}


\begin{abstract}
In this work, we present basic results and applications of Stepanov pseudo almost periodic functions with measure. Using only the continuity assumption, we prove a new composition result of $\mu$-pseudo almost periodic functions in Stepanov sense. Moreover, we present different applications to semilinear differential equations and inclusions in Banach spaces with weak regular forcing terms. We prove the existence and uniqueness of $\mu$-pseudo almost periodic solutions (in the strong sense) to a class of semilinear fractional inclusions and semilinear evolution equations, respectively, provided that the nonlinear forcing terms are only Stepanov $ \mu $-pseudo almost periodic in the first variable and not a uniformly strict contraction with respect to the second argument. Some examples illustrating our theoretical results are also presented.

\end{abstract}

\maketitle
\section{Introduction}
The notion of almost periodicity was introduced by Harald Bohr around 1925 and later generalized by many others in different contexts (\cite{H.Bohr}); namely, almost automorphic functions due to Bochner \cite{Boch2}, Stepanov almost periodic functions \cite{Step} and pseudo almost periodic functions \cite{Zh} are different interesting generalizations. Unlike the classical almost periodic functions, Stepanov almost periodic functions are only locally integrable and not necessarily bounded or continuous. Recently, in \cite{Ezz1}, the authors have introduced the concept of $ \mu $-pseudo almost periodic functions defined as perturbation of an almost periodic function by an ergodic term (i.e., a bounded continuous functions with mean vanishing at infty, see Section \ref{Section2}), while in this paper, the ergodicity is defined through positive measures in which the previous concepts of pseudo almost periodic functions and weighted pseudo almost periodic functions \cite{Diag2} are just particular cases. Lately, in \cite{Moi,Ess}, the authors introduced a more general concept of $\mu$-pseudo almost periodic functions in Stepanov sense, saying, Stepanov almost periodic functions perturbed by general locally integrable ergodic terms, see Section \ref{Section2} for more details.\\

In the literature, we found several works devoted to the existence and uniqueness of $\mu$-pseudo almost periodic solutions to semilinear evolution equations and inclusions; we quote \cite{AkdSou,Moi,Moi2,Ezz1,Cieutat,Pinto3,Diag2,Diag3,Moi4,Ess,nova-mono,LZ,Li,Pinto1,Pinto2,Vu,Wei,Zh}. A solution (at least in a mild sense) is usually represented by an integral operator (this holds under a certain decay of the associated linear system). More specifically, for given an integral operator solution, namely,
\begin{equation}
(\mathcal{S} u)(t)=\int_{\mathbb{R}} \mathcal{G}(t,s)f(s,u(s)) ds, \quad t \in \mathbb{R}, \label{operator solution}
\end{equation}
where the input parameters, $ (\mathcal{G}(t,s))_{t\geq s} $ which represents the Green function (resp. the resolvent operator), it is assumed to be bi-almost periodic, and $f$ is the nonlinearity. In the case where $f$ is uniformly Lipschitzian with respect to the second variable and $\mu$-pseudo almost periodic (resp.  Stepanov $\mu$-pseudo almost periodic) in $t$, it was shown (see \cite{AkdSou,Moi,Ezz1,Cieutat,Pinto3,Diag2,Ess,Li,Wei}) that the output (the corresponding solution) $u$  is $\mu$-pseudo almost periodic. Extensively, the existence and uniqueness results established so far are obtained in view of some composition results and the Banach strict contraction principle. However, in the instance where the input $f$ belongs to a broader class of Stepanov $\mu$-pseudo almost periodic functions and not necessarily uniformly Lipschitzian, there are no consistent results in the literature devoted to the existence (and/or the uniqueness) to $\mu$-pseudo almost periodic solutions, since the composition results for Stepanov $\mu$-pseudo almost periodic functions established in the literature up to here require the uniform Lipschitz condition.\\


The aim of this work is to extend the literature's results and prove a new composition result of Stepanov $\mu$-pseudo almost periodic functions using only the continuity condition (more specifically, we do not use the uniform Lipschitz assumption). That is, for a given function $f:\mathbb{R}\times X \longrightarrow Y $ which is Stepanov $\mu$-pseudo almost periodic with respect to $t$ and continuous with respect to $x$  and for a given a function $u:\mathbb{R}\longrightarrow X$ which is $\mu$-pseudo almost periodic for any Banach spaces $X$ and $Y$, $ f(\cdot,u(\cdot)) $ is Stepanov $\mu$-pseudo almost periodic.  Afterward, we take advantage of this main result and prove the existence and uniqueness of solutions to a class of semilinear fractional inclusions and a class of semilinear nonautonomous evolution equations in the Banach spaces, respectively. In fact, using Meir and Keeler fixed point argument \cite{kiler}, we prove the existence and uniqueness of $\mu$-pseudo almost periodic solutions to the following inclusion:
\begin{align}\label{mqwert12345}
D_{t,+}^{\gamma}u(t)\in {\mathcal A}u(t)+f(t,u(t)),\ t\in {\mathbb R},
\end{align}
where $D_{t,+}^{\gamma}$ denotes the Riemann-Liouville fractional derivative of order $\gamma \in (0,1),$ 
$f : {\mathbb R} \times X \rightarrow X$ is Stepanov $ \mu  $-pseudo almost periodic in $t \in \mathbb{R}$ and satisfies certain properties with respect to $x \in X$. Here, the operator solution \eqref{operator solution} has the following form:
\begin{equation}
(\mathcal{S} u)(t)=\int_{-\infty}^{t} R_{\gamma}(t-s)f(s,u(s)) ds, \quad t \in \mathbb{R}, \label{operator solution1}
\end{equation}
where $(R_{\gamma}(t))_{t\geq 0}$ is the associated operator resolvent. The main result obtained here does not require that $\mathcal{S}$ be necessarily a uniformly strict contraction mapping, see Theorem \ref{mnemojwer2}. As application, we study a class of semilinear fractional Poisson heat equations in $ L^2 $-setting,  see the model \eqref{model App1}. Furthermore, by the 'local' Banach contraction principle, we prove the existence and uniqueness of $\mu$-pseudo almost periodic solutions (see Theorem \ref{Main Theorem2 Eq1} for $ 1<p<\infty $ and Theorem \ref{Main Theorem3 Eq1} for $p=1$) to the following evolution equation: 
\begin{equation}
     x'(t)= A(t)x(t) + f(t, x(t))\quad \text{for}\; t \in \mathbb{R}, \label{Eq_1}    
\end{equation}
where $(A(t),D(A(t))), \; t\in \mathbb{R}$ is a family of closed linear operators on a Banach space $X$ that generates an evolution family $ (U(t,s))_{t\geq s} $ which have an exponential dichotomy on $\mathbb{R}$,  the nonlinear term $f: {\mathbb R} \times X \rightarrow X$ is assumed to be Stepanov $ \mu$-pseudo almost periodic (of order $ 1\leq p<\infty$) in $t \in \mathbb{R}$ and only Lipschitzian in bounded sets with respect to $x \in X$. Besides, we provide an application to a class of reaction-diffusion equations describing the behavior of bounded solutions of a one-species intraspecific competition Lotka--Volterra model.\\

The work is organized as follows. In Section \ref{Section2}, we provide preliminary results on Stepanov $ \mu  $-pseudo almost periodic functions; Section \ref{Section3} is devoted to the study of a new composition result of Stepanov $ \mu$-pseudo almost periodic functions. In Section \ref{Section4}, we prove the existence and uniqueness of $ \mu  $-pseudo almost periodic solutions to the inclusion \eqref{mqwert12345} (see Section \ref{Section41}) and the equation \eqref{Eq_1}   (see Section \ref{Section42}), respectively.

\subsection*{Notations}
Throughout this work, $(X,\|\cdot\|) $ and $(Y,\|\cdot\|_{Y})$ are two complex Banach spaces; $(\mathcal{L}(X),\|\cdot\|_{\mathcal{L}(X)})$ stands for the Banach algebra of bounded linear operators in $ X $. $BC(\mathbb{R},X)$ equipped with the sup norm  $ \| \cdot \|_{\infty} $, is the Banach space of bounded continuous functions $f:\mathbb{R} \rightarrow X$. Moreover, for $ 1\leq p <\infty $, $ q $ denotes its conjugate exponent defined by $ 1/p +1/q=1,$ if $ p\neq  1,$ and $ q=\infty ,$ if $ p=1 $. By  $ L^{p}_{loc}(\mathbb{R},X)$ (resp. $ L^{p}(\mathbb{R},X)$), we designate the space (resp. the Banach space) of all equivalence classes of measurable functions $f:\mathbb{R} \rightarrow X$ such that $\|f(\cdot)\|^{p}$ is locally integrable (resp. integrable). By $ \rho(A) $ we denote the resolvent set of $(A,D(A)),$ defined by 
$ \rho(A):= \lbrace \lambda \in \mathbb{C}: \, (\lambda -A)^{-1}\; \text{exists in}\; \mathcal{L}(X)\rbrace \subseteq \mathbb{C};$
the spectrum $ \sigma(A) $ of $(A,D(A))$ is defined by $ \sigma(A):= \mathbb{C}\setminus \rho(A) $. For $\lambda \in \rho(A)$, the resolvent operator $R(\lambda, A)$ is defined by $ R(\lambda, A) := (\lambda - A)^{-1}.  $ By $ \mathcal{B}(\mathbb{R}) $ we denote the
collection of all Lebesgue measurable subsets of $\mathbb{R}$ and by $ \mathcal{M} $ we denote the set of all positive measures $ \mu $ on $ \mathcal{B}(\mathbb{R}) $ satisfying $ \mu(\mathbb{R})=+\infty $ and $ \mu(\left[a,b\right] )< +\infty $ for all $ a,\ b \in \mathbb{R} $ with $ a\leq b $.
\section{Preliminaries}\label{Section2}
In this section, we provide preliminaries about the concept of $ \mu $-pseudo almost periodic functions, in Bohr and Stepanov sense respectively, needed to elaborate the main results of this work.

\subsection{Almost periodic functions}\label{section21}
A continuous function $ f:\mathbb{R}\longrightarrow X $ is said to be almost periodic if, for every $\varepsilon > 0,$ there exists $ l_{\varepsilon}  > 0 $ such that for every $ a\in \mathbb{R} $, there exists $ \tau \in \left[a,a+ l_{\varepsilon} \right]  $ satisfying $ \| f(t+\tau)-f(t) \| < \varepsilon ,$ $ t\in \mathbb{R} .$ The space of all such functions is denoted by $ AP(\mathbb{R},X) .$

Let $ 1\leq p< \infty $.  A function $ f\in L^{p}_{loc} (\mathbb{R},X) $ is said to be Stepanov $p$-bounded if $$ \displaystyle \sup_{t\in \mathbb{R}}  \left( \int_{\left[ t,t+1\right] } \|f(s)\|^{p}ds\right) ^{\frac{1}{p}}=\displaystyle \sup_{t\in \mathbb{R}}  \left( \int_{\left[ 0,1\right] } \|f(t+s)\|^{p}ds\right) ^{\frac{1}{p}} <\infty.$$ 
The space of all such functions is denoted by $ BS^{p} (\mathbb{R},X);$ equipped with the norm
\begin{eqnarray*} 
   \|f\|_{BS^{p} }&:=& \displaystyle \sup_{t\in \mathbb{R}}  \left( \int_{\left[ t,t+1\right] } \|f(s)\|^{p}ds\right) ^{\frac{1}{p}} = \sup_{t\in \mathbb{R}}\|f(t+\cdot)\|_{L^{p} (\left[ 0,1\right],X)} ,
\end{eqnarray*}   
 $ BS^{p} (\mathbb{R},X)$ is a Banach space.
The following inclusions hold:
\begin{eqnarray}
BC(\mathbb{R},X) \subseteq BS^p (\mathbb{R},X) \subseteq L^{p}_{loc} (\mathbb{R},X) . \label{incls}
\end{eqnarray} 

\begin{definition}[Bochner transform] 
Let $f\in L^{p}_{loc} (\mathbb{R},X) $ for $ 1\leq p< \infty $. The Bochner transform of $f(\cdot)$ is the function $f^b :\mathbb{R}\longrightarrow L^{p} (\left[ 0,1\right],X) ,$ defined by $$ (f^b (t))(s):= f(t+s) \quad \mbox{for} \; s\in \left[ 0,1\right], \; t\in \mathbb{R}.$$
\end{definition}

 Now, we recall the definition of Stepanov $p$-almost periodicity:
\begin{definition} Let $ 1\leq p< \infty $. A function $ f \in L^{p}_{loc} (\mathbb{R},X) $ is said to be Stepanov $p$-almost periodic ($S^p$-almost periodic, for short), if for every $\varepsilon > 0,$ there exists $ l_{\varepsilon}  > 0 $, such that for every $ a\in \mathbb{R} $, there exists $ \tau \in \left[a,a+ l_{\varepsilon} \right]  $ satisfying $$ \displaystyle \left( \int_{\left[ t,t+1\right] }\| f(s+\tau)-f(s) \|^{p}ds\right)^{\frac{1}{p}}  < \varepsilon \quad \mbox{for all }\; t\in \mathbb{R}.$$ The space of all such functions is denoted by $ APS^{p}(\mathbb{R},X) .$ 
\end{definition}

\begin{remark}\textbf{\cite{Amerio,Moi,nova-mono}}\label{RemCompSpAp}
\newline

\textbf{(i)} Every (Bohr) almost periodic function is $S^p$-almost periodic for $ 1\leq p< \infty $. The converse is not true in general.\\

\textbf{(ii)} For all $ 1\leq p_1 \leq p_2 < \infty $, if $f$ is $S^{p_2}$-almost periodic, then $f$ is $S^{p_1}$-almost periodic. \\

\textbf{(iii)} The Bochner transform of an $X $-valued function is an $L^{p} (\left[ 0,1\right],X) $-valued function. Moreover, a function $f (\cdot)$ is $S^p$-almost periodic if and only if $f^{b}(\cdot)$ is (Bohr) almost periodic. \\

\textbf{(iv)} A function $ \varphi(t,s)$, defined for $t\in \mathbb{R},s\in [0,1]$ is the Bochner transform of a function $f(\cdot)$ (i.e., $ \exists$ $f:\mathbb{R}\longrightarrow X$ such that $(f^b (t))(s)=\varphi(t,s),\; t\in \mathbb{R}, s\in \left[ 0,1\right]$) if and only if $ \varphi(t+\tau,s-\tau)=\varphi(t,s)$ for all $t\in \mathbb{R}, s\in \left[ 0,1\right]$ and $\tau \in [s-1,s].$ 
\end{remark} 

A sufficient condition for a Stepanov almost periodic function to be Bohr almost periodic is given in the next theorem:

\begin{theorem}\textbf{\cite{Amerio}}\label{PropCovApS}
Let  $f\in L^{p}_{loc} (\mathbb{R},X) $ for $ 1\leq p< \infty $. If $f(\cdot)$ is $S^p$-almost periodic and uniformly continuous, then $f$ is almost periodic. 
\end{theorem}
\begin{definition}
Let $ 1\leq p< \infty $. A  function $ f: \mathbb{R}\times X\longrightarrow Y$ such that $f(\cdot,x)\in L^{p}_{loc} (\mathbb{R},Y) $  for each $ x\in X $ is said to be almost periodic in $ t ,$ uniformly with respect to $ x $ in $X$ if for each compact set $K$ in $X$ and for each $\varepsilon > 0 $  there exists $l_{\varepsilon,K} > 0 $, such that for every $ a \in \mathbb{R}$ there exists $ \tau \in \left[a,a+l_{\varepsilon,K} \right]  $ satisfying: $$ \displaystyle \sup_{x\in K}\left( \int_{\left[ t,t+1\right] }\| f(s+\tau,x)-f(s,x) \|^{p}ds\right)^{\frac{1}{p}}  < \varepsilon \quad \mbox{for all}\; t\in \mathbb{R}.$$ 
The space of all such functions is denoted by $ APS^{p}U(\mathbb{R}\times X,Y) .$ 
\end{definition}
\subsection{$\mu$-pseudo almost periodic functions}\label{section22}
In this subsection, we provide the main properties of $ \mu$-ergodic functions and (Stepanov) $\mu$-pseudo almost periodic functions. We will use the following assumption on the measure $\mu$: \\

 \textbf{(M)}  For all $ \tau \in \mathbb{R} $, there exist a number $ \beta >0 $ and a bounded interval $ I $ such that $$ \mu(\left\lbrace a+\tau: a\in A\right\rbrace ) \leq \beta\mu(A), \qquad \mbox{provided }\, A\in \mathcal{B}(\mathbb{R})\; \mbox{ and }\; A\cap I=\emptyset . $$ 
\begin{remark} In particular, the Lebesgue measure $ \lambda $ on $ (\mathbb{R},  \mathcal{B}(\mathbb{R}))$ belongs to $ \mathcal{M} $ and it satisfies the hypothesis \textbf{(M)}.
\end{remark}
\begin{definition}\textbf{\cite{Ezz1}}
 Let $\mu \in \mathcal{M} $. A bounded continuous function $ f:\mathbb{R}\longrightarrow X $ is said to be $ \mu$-ergodic if $$ \displaystyle \lim_{r\rightarrow +\infty } \frac{1}{\mu(\left[-r,r \right] )} \int_{\left[-r,r \right]} \|f(t)\|d\mu(t)=0. $$ 
 The space of all such functions is denoted by $ \mathcal{E}(\mathbb{R},X,\mu). $
\end{definition}

%
\begin{examples}\textbf{\cite{Ezz1}}
\newline
\textbf{(1)} Any ergodic function which belongs to the space $PAP_{0}({\mathbb R},X),$ introduced by C. Zhang in \cite{Zh}, is nothing else but a $ \mu $-ergodic function in the particular case when $ \mu $ is the Lebesgue measure. \\
\textbf{(2)} Let $ \rho:\mathbb{R}\longrightarrow \left[ 0, + \infty\right)  $ be a  Lebesgue measurable function. We define the positive measure $ \mu $ on $ \mathcal{B}(\mathbb{R}) $ by $$ \mu(A):= \int_{A} \rho(t) dt \qquad \mbox{for}\; A\in  \mathcal{B}(\mathbb{R}), $$ where  $ dt $ denotes the Lebesgue measure. The measure $ \mu $ is absolutely continuous with respect to $ dt $ and the function $ \rho $ is called the Radon-Nikodym derivative of $ \mu $ with respect to $ dt $. In that case $ \mu \in \mathcal{M} $ if and only if the function $ \rho $ is locally Lebesgue-integrable on $\mathbb{R}$ and satisfies $$ \int_{\mathbb{R}} \rho(t) dt=+\infty. $$
\end{examples}
\begin{definition}\textbf{\cite{Ess}}
 Let $\mu \in \mathcal{M} $. A function  $ f\in BS^{p} (\mathbb{R},X)$ is said to be Stepanov $ \mu$-ergodic ($\mu$-$S^{p}$-ergodic, for short) if 
\begin{align}
\notag  \lim_{r\rightarrow +\infty } &\frac{1}{\mu(\left[-r,r \right] )} \int_{\left[-r,r \right]} \left( \int_{\left[ t,t+1\right] }\| f(s) \|^{p}ds\right)^{\frac{1}{p}}\ d\mu(t)
\\\label{StepErg}&=\displaystyle \lim_{r\rightarrow +\infty } \frac{1}{\mu(\left[-r,r \right] )} \int_{\left[-r,r \right]} \| f^{b}(t)\|_{p,X}\ d\mu(t)=0. 
\end{align}
The space of all such functions is denoted by $ \mathcal{E}^{p}(\mathbb{R},X,\mu). $
\end{definition}

\begin{remark} Using \eqref{StepErg}, we obtain that $ f\in \mathcal{E}^{p}(\mathbb{R},X,\mu) \;  \mbox{if and only if } \; f^{b}\in \mathcal{E}(\mathbb{R},L^{p} (\left[ 0,1\right],X),\mu)  .$
\end{remark}

\begin{proposition}\textbf{\cite{Ess}} \label{Trans inva Ep}  Let  $ 1\leq p< \infty $ and $\mu \in \mathcal{M} $ satisfy  \textbf{(M)}. Then, the following holds:
\begin{itemize}
\item[\textbf{(i)}]   $\mathcal{E}^{p}(\mathbb{R},X,\mu)$ is translation invariant. 
\item[\textbf{(ii)}] $\mathcal{E}(\mathbb{R},X,\mu)\subseteq\mathcal{E}^{p}(\mathbb{R},X,\mu)$.
\end{itemize}
\end{proposition}
%

\begin{definition}\label{DefErgUChar}
Let $ \mu \in \mathcal{M} $ and $1\leq p<\infty.$ A function $ f:\mathbb{R}\times X\longrightarrow Y $ such that $ f(\cdot,x) \in BS^{p}(\mathbb{R},Y) $ for each $ x\in X $ is said to be $ \mu $-$ S^{p} $-ergodic in $t$  with respect to $x$ in $X$  if the following holds:
\begin{itemize}
\item[\textbf{(i)}]  For all $ x\in X,  \, f(\cdot ,x) \in \mathcal{E}^{p}(\mathbb{R},Y,\mu) $. 
\item[\textbf{(ii)}] $ f $ is $ S^{p} $-uniformly continuous with respect to the second argument on each compact subset $K$ in $X$, namely: for every $ \varepsilon>0 $ there exists $ \delta_{K, \varepsilon} >0$ such that  
for all $ x_{1}, x_{2} \in K $, we have 
\begin{eqnarray}
\| x_{1}-x_{2}\| \leq \delta_{K, \varepsilon} \Longrightarrow \left( \int_{t}^{t+1}\|f(s,x_{1})-f(s,x_{2}) \|^{p}ds\right) ^{\frac{1}{p}} \leq \varepsilon \quad \text{for all}\; t\in \mathbb{R}. \label{ErgSpChar}
\end{eqnarray}
\end{itemize}
Denote by  $ \mathcal{E}^{p}U(\mathbb{R}\times X,Y,\mu)  $ the set of all such functions.
\end{definition}

Now, we give the important properties of $\mu$-pseudo almost periodic functions and $ S^{p} $-$\mu$-pseudo almost periodic functions.
\begin{definition}\textbf{\cite{Ezz1}}
Let $ \mu \in \mathcal{M} $. A continuous function $ f:\mathbb{R}\longrightarrow X $ is said to be $ \mu $-pseudo almost periodic  if $ f $ can be decomposed in the form
$$ f=g+\varphi ,$$ where $ g \in AP (\mathbb{R},X) $ and  $\varphi \in \mathcal{E} (\mathbb{R},X,\mu) .$
The space of all such functions is denoted  by $ PAP(\mathbb{R},X,\mu) .$ 
\end{definition}

\begin{proposition}\textbf{\cite{Ezz1}}
Let  $ \mu \in \mathcal{M} $ satisfy $\textbf{(M)}$. Then, the following holds:
\begin{itemize}
\item[\textbf{(i)}] The decomposition of a $\mu$-pseudo almost periodic in the form  $f=g+\varphi ,$ where $ g \in AP (\mathbb{R},X) $ and $\varphi \in \mathcal{E} (\mathbb{R},X,\mu) ,$ is unique.
\item[\textbf{(ii)}] $PAP (\mathbb{R},X,\mu) $ equipped with the sup-norm is a Banach space.
\item[\textbf{(iii)}] $PAP (\mathbb{R},X,\mu) $ is translation invariant.
\end{itemize}
\end{proposition}
\begin{definition}
Let $ \mu \in \mathcal{M} $. A function $ f:\mathbb{R}\times X\longrightarrow Y $ such that $ f(\cdot,x) \in BS^{p}(\mathbb{R},Y) $ for each $ x\in X $ is said to be $ S^{p} $-$ \mu $-pseudo almost periodic if $ f $ can be decomposed in the form
$$ f=g+\varphi, $$ where $ g \in APS^{p}U (\mathbb{R}\times X,Y) $ and  $\varphi \in \mathcal{E} ^{p}U(\mathbb{R}\times X,Y,\mu) .$\\
The space of all such functions will be denoted by $ PAPS^{p}U(\mathbb{R},X,\mu) .$ 
\end{definition}
Furthermore, we need the following preliminary results obtained in  \cite{kamal}.
\begin{lemma}\label{LemApSpCom1Chapter1}
Let $ 1\leq p < +\infty $ and $ f:\mathbb{R}\times X\longrightarrow Y $ be such that $ f(\cdot, x) \in L^{p}_{loc}(\mathbb{R}, Y) $ for each $ x\in X .$ Then, $ f\in APS^{p}U(\mathbb{R}\times X,Y) $ if and only if the following holds:
\begin{itemize}
\item[\textbf{(i)}] For each $ x\in X $, $ f(\cdot, x) \in APS^{p}(\mathbb{R},Y) .$
\item[\textbf{(ii)}] $ f $ is $ S^{p} $-uniformly continuous with respect to the second argument on each compact subset $K$ in $X$ in the following sense: for all $ \varepsilon>0 $ there exists $ \delta_{K, \varepsilon} >0$ such that  
for all $ x_{1}, x_{2} \in K $ one has 
\begin{eqnarray}
\| x_{1}-x_{2}\| \leq \delta_{K, \varepsilon} \Longrightarrow \left( \int_{t}^{t+1}\|f(s,x_{1})-f(s,x_{2}) \|^{p}ds\right) ^{\frac{1}{p}} \leq \varepsilon \quad \text{for all} \; t\in \mathbb{R}  . \label{ComCharSpChapter1}
\end{eqnarray}
\end{itemize}
\end{lemma}

It is clear that Lemma \ref{LemApSpCom1Chapter1} implies the following:

\begin{proposition}\label{PropApSpCom1Chapter1}
Let $ \mu \in \mathcal{M} $ and  $f \in PAPS^{p}U(\mathbb{R}\times X,Y,\mu) $, for $ 1\leq p < +\infty $. Then, the following holds:
\begin{itemize}
\item[\textbf{(i)}] for each $ x\in X $, $ f(\cdot, x) \in PAPS^{p}(\mathbb{R},Y,\mu) ;$
\item[\textbf{(ii)}] $ f $ is $ S^{p} $-uniformly continuous with respect to the second argument on each compact subset $K$ in $X;$ namely, for each $ \varepsilon>0 $ and for each compact set $K$ in $X$ there exists $ \delta_{K, \varepsilon} >0 $ such that  
for all $ x_{1},\ x_{2} \in K $, we have 
\begin{eqnarray}
\| x_{1}-x_{2}\| \leq \delta_{K, \varepsilon} \Longrightarrow \left( \int_{t}^{t+1}\|f(s,x_{1})-f(s,x_{2}) \|^{p}ds\right) ^{\frac{1}{p}} \leq \varepsilon \; \text{for all} \; t\in \mathbb{R} . \label{ComCharSpChapter1}
\end{eqnarray}
\end{itemize}
\end{proposition}
Next, we provide some examples of Stepanov $\mu$-pseudo almost periodic functions of order $ 1\leq p<\infty $.
\begin{examples}\label{prijem}
Let $X$ be any Banach space and let $\mu $ be a measure with the Radon–Nikodym derivative $\theta$ defined by
\begin{equation}
\theta(t)=\left\{
\begin{aligned}
e^{t} & \; \text{for}\; t \leq 0,\\
 1 & \; \text{for}\; t > 0 .
\end{aligned} \label{theta}
\right. 
\end{equation} 
From \cite[Example 3.6]{Ezz1}, the measure $\mu$ satisfies the hypothesis \textbf{(M)}. 
Consider the function $ \Phi: \mathbb{R}\longrightarrow \mathbb{R}$ given by $ \Phi (t):= \Phi_{1}(t)+\Phi_2(t)  $ with $ \Phi_2(t)= \left( \arctan(t)-\dfrac{\pi}{2} \right) $ and  $$ \displaystyle \Phi_{1}(t)= \sum_{n\geq 1} \beta_{n}(t), $$ such that, for every $ n \geq 1 $ $$ \beta_{n}(t)= \displaystyle \sum_{i \in P_{n}} H(n^{2}(t-i)),$$
with $ P_{n}=3^{n}(2\mathbb{Z}+1) $ and $ H \in C_{0}^{\infty}(\mathbb{R},\mathbb{R}) $ with support in $ (\frac{-1}{2},\frac{1}{2}) $ such that 
$$  H \geq 0 , \quad H(0)=1 \quad \text{and} \quad \int_{\frac{-1}{2}}^{\frac{1}{2}} H(s) ds =1 .  $$
By the proof in \cite[Section 5]{Ezz1}, the function $ t\longmapsto \arctan(t)-\dfrac{\pi}{2} $ belongs to $ \mathcal{E}(\mathbb{R},\mathbb{R},\mu) $. Otherwise, from \cite{Tar},  $\Phi_{1} \in C^{\infty}(\mathbb{R},\mathbb{R}) $, but, $ \Phi_{1}  \notin AP(\mathbb{R}) $ since it is not bounded. However, $\Phi_{1}  \in APS^{1}(\mathbb{R}) .$ \\

Let $\Psi_{1}: {\mathbb R} \rightarrow {\mathbb C}$ be any essentially bounded function which is $S^{p}$-almost periodic ($1\leq p<\infty$) but not almost periodic; for example, if $\alpha, \ \beta \in {\mathbb R}$ and $\alpha \beta^{-1}$ is a well-defined irrational number, then we can take
$$
\Psi_1(t)=\sin\Biggl(\frac{1}{2+\cos \alpha t + \cos \beta t}\Biggr),\quad t\in {\mathbb R}
$$
or
$$
\Psi_1(t)=\cos\Biggl(\frac{1}{2+\cos \alpha t + \cos \beta t}\Biggr),\quad t\in {\mathbb R}.
$$
The function $ \Psi_1$ is bounded continuous but not uniformly continuous, see \cite{BaGu} for more details. Set $ \Psi: \mathbb{R}\longrightarrow \mathbb{C}$ defined such that $ \Psi(t)=\Psi_1(t)+\Phi_2(t) $, $ t\in \mathbb{R} $. Thus, $\Psi $ yields a $S^{p}$-$\mu$-pseudo almost periodic scalar function (for all $1\leq p <\infty$). Further on, let $h: X\rightarrow X$ be any continuous function. Then the functions $f(t,x):=\Psi(t)h(x)$ and $ g(t,x):=\Phi(t)h(x) $ for $t\in {\mathbb R}$ and $x\in X$  define two examples of (purely) Stepanov $\mu$-pseudo almost periodic $X$-valued functions. In particular, $ f $ is  $ S^{1} $-$\mu$-pseudo almost periodic and $ g $ is  $ S^{2} $-$\mu$-pseudo almost periodic.
\end{examples}


\section[New composition results]{New composition results for Stepanov $\mu$-pseudo almost periodic functions}\label{Section3}
In this section we establish new composition results of $\mu$-pseudo almost periodic functions in Stepanov sense of order $ 1\leq p<\infty $. 
\begin{theorem}\label{ThmComApSpChapter1} (\cite{kamal})
Let $ 1\leq p < +\infty $ and $ f\in APS^{p}U(\mathbb{R}\times X,Y) $. Assume that $x \in AP (\mathbb{R},X)$. Then, $ f(\cdot,x(\cdot)) \in APS^{p}(\mathbb{R}, Y).$ 
\end{theorem}
\begin{remark}\label{okee}
In \cite[Theorem 2.7.2]{nova-mono}, we have used condition that $x \in APS^{p}(\mathbb{R},X)$ and there exists a set ${\mathrm E} \subseteq \mathbb{R}$ with $m ({\mathrm E})= 0$ such that the set
$ K =\{x(t) : t \in \mathbb{R} \setminus {\mathrm E}\}$
is relatively compact in $ Y$ (this condition is clearly satisfied if $x \in AP (\mathbb{R},X);$ here $m(\cdot)$ denotes the Lebesgue measure). In this case, we have $f(\cdot, x(\cdot)) \in APS^{p}({\mathbb R} , Y)$ provided that there exists a finite Lipschitz constant $L\geq 1$ such that 
\begin{align*}
\| f(t,x)-f(t,y)\| \leq L\|x-y\|,\quad t\in \mathbb{R},\ x,\ y\in X.
\end{align*}
As Theorem \ref{ThmComApSpChapter1} shows, we do not need this Lipschitz type assumption for $x \in AP(\mathbb{R},X).$
\end{remark}

In order to prove our main composition Theorem of Stepanov $\mu$-pseudo almost periodic functions, we need to prove the following substantial preliminary results.  
\begin{theorem}\label{ThmLemErgSpComChapter1}
Let $ 1\leq p < +\infty $ and $\mu \in \mathcal{M}$. If $f \in \mathcal{E}^{p}U (\mathbb{R}\times X,Y,\mu) $ and $x \in BC (\mathbb{R},X)$ such that $ K= \overline{\lbrace x(t) : t \in \mathbb{R}} \rbrace $ is compact in $ X $ (i.e., with relatively compact range), then $f(\cdot,x(\cdot)) \in \mathcal{E}^{p} (\mathbb{R},Y,\mu) $.
\end{theorem}  
\begin{proof}
Let $f \in \mathcal{E}^{p}U (\mathbb{R}\times X,Y,\mu) $ and $ K= \overline{\lbrace x(t) : t \in \mathbb{R}} \rbrace \subset X$ a compact subset.  Then, for every $\varepsilon >0,$ there exists $ \delta_{\varepsilon,K}>0 $ such that \eqref{ComCharSpChapter1} holds. Since $K$ is compact,  there exists a finite subset $ \lbrace x_{1},...,x_{n} \rbrace  \subseteq K$ ($ n\in \mathbb{N} $) such that $ K\subseteq \bigcup_{i=1}^{n} B(x_{i}, \delta_{K, \varepsilon} ) .$ Therefore, for every $t\in \mathbb{R} $, there exists $i(t)\in \{1,...,n\} $ such that $ \|x(t)-x_{i(t)}\| \leq \delta .$ Furthermore,
\begin{eqnarray}
\left( \int_{t}^{t+1}\| f(s ,x(s))\|_{Y}^{p}ds\right) ^{\frac{1}{p}}& \leq & \nonumber \left(\int_{t}^{t+1}\| f(s ,x(s))-f(s , x_{i(t)}) \|_{Y}^{p}ds\right) ^{\frac{1}{p}}\\ & &+ \left( \int_{t}^{t+1}\| f(s,x_{i(t)} ) \|_{Y}^{p}ds\right) ^{\frac{1}{p}} \nonumber  \\ &\leq & \varepsilon + \sum_{i=1}^{n}\left( \int_{t}^{t+1}\| f(s,x_{i} ) \|_{Y}^{p}ds\right) ^{\frac{1}{p}},\ t\in \mathbb{R} . \label{ComResErFor1Chapter1}
\end{eqnarray}
Since $f(\cdot ,x_{i} ) \in \mathcal{E}^{p} (\mathbb{R},Y,\mu) $ for $ i=1,...,n $,  we have 
\begin{eqnarray*}
& &\frac{1}{\mu(\left[-r,r \right])}\int_{-r}^{r}\left( \int_{t}^{t+1}\| f(s ,x(s))\|_{Y}^{p}ds\right) ^{\frac{1}{p}}d\mu(t)\\ &  \leq  &\varepsilon  + \frac{1}{\mu(\left[-r,r \right])}\sum_{i=1}^{n}\int_{-r}^{r}\left( \int_{t}^{t+1}\| f(s,x_{i} ) \|_{Y}^{p}ds\right) ^{\frac{1}{p}} d\mu(t),\label{ComResErFor2Chapter1}
\end{eqnarray*}
for $ r  >0$ large enough.
Consequently,
\begin{eqnarray}
\limsup_{r\rightarrow +\infty}\frac{1}{\mu(\left[-r,r \right])}\int_{-r}^{r}\left( \int_{t}^{t+1}\| f(s ,x(s))\|_{Y}^{p}ds\right) ^{\frac{1}{p}}d\mu(t) &\leq & \varepsilon .  \label{ComResErFor3Chapter1}
\end{eqnarray}
Since $ \varepsilon >0 $ was arbitrary, \eqref{ComResErFor3Chapter1} yields
\begin{eqnarray*}
\lim_{r\rightarrow +\infty}\frac{1}{\mu(\left[-r,r \right])}\int_{-r}^{r}\left( \int_{t}^{t+1}\| f(s ,x(s))\|_{Y}^{p}ds\right) ^{\frac{1}{p}}d\mu(t)=0.  
\end{eqnarray*}
\end{proof}
\begin{corollary}\label{LemErgSpComChapter1}
Let $\mu \in \mathcal{M}$. Assume that $x \in AP (\mathbb{R},X)$ and $f \in \mathcal{E}^{p}U (\mathbb{R}\times X,Y,\mu) $. Then, $f(\cdot,x(\cdot)) \in \mathcal{E}^{p} (\mathbb{R},Y,\mu) $.
\end{corollary} 
\begin{proof}
From $ x \in AP (\mathbb{R},X) $, we deduce that $ x \in BS^{p} (\mathbb{R},X)  $ and $ K= \overline{\lbrace x(t) : t \in \mathbb{R}} \rbrace $ is a compact subset of $ X $. Hence, conditions and hypotheses of Theorem \ref{ThmLemErgSpComChapter1} are satisfied. 
\end{proof} 
\begin{lemma}\label{lemmaErgChapter1}\textbf{\cite{Ezz1}}
Let $\mu \in \mathcal{M}$ and $ f\in BC(\mathbb{R},X).$ Then, $ f \in \mathcal{E}(\mathbb{R},X,\mu) $ if and only if for all $ \varepsilon  > 0 $ 
\begin{eqnarray}
\lim_{r\rightarrow +\infty} \dfrac{\mu \left(M_{\varepsilon, r} (f)  \right) }{\mu([-r,r])} =0, \label{ComResErFor1} 
\end{eqnarray}
where $ M_{\varepsilon, r} (f) :=\left\lbrace t\in [-r,r]:\| f(t)\| \geq \varepsilon \right\rbrace $.
\end{lemma}  
The proof of our result related to the composition of $ S^{p} $-$\mu$-pseudo almost periodic functions is based on the following lemma due to Schwartz \cite[p. 109]{Schw}.
\begin{lemma}\label{LemSchChapter1}
Let $\Phi \in C(X,Y)$. Then, for each compact set $K \subseteq X$ and for each $\varepsilon > 0$, there exists $\delta_{K,\varepsilon} > 0$ such that for any $x_{1},\ x_{2} \in X$, we have
\begin{eqnarray*}
x_{1} \in K \quad \text{and} \quad \| x_{1}-x_{2} \|\leq \delta_{K,\varepsilon} \Rightarrow \| \Phi( x_{1})-\Phi(x_{2}) \|_{Y} \leq \varepsilon .
\end{eqnarray*}
\end{lemma} 

\begin{theorem}\label{CompResMuChapter1}
Let $ 1\leq p < +\infty $ and $ \mu \in \mathcal{M} $. Assume the following:
\begin{itemize}
\item[(i)] $ f:\mathbb{R}\times X\longrightarrow Y $ be a function such that $ f=\tilde{f}+\varphi \in PAPS^{p}U(\mathbb{R}\times X,Y,\mu)$ with $ \tilde{f}\in APS^{p}U(\mathbb{R}\times X,Y) $ and $ \varphi \in \mathcal{E}^{p}U(\mathbb{R}\times X,Y,\mu) $; 
\item[(ii)] $x=x_{1}+x_{2} \in PAP(\mathbb{R},X,\mu) ,$ where $ x_{1} \in AP(\mathbb{R},X) $ and $x_{2} \in \mathcal{E}(\mathbb{R},X,\mu)$;
\item[(iii)] for every bounded subset $B\subseteq X$, we have $\sup_{x \in B}\|f(\cdot,x)\|_{BS^{p}} < \infty .$ 
\end{itemize}
Then, $ f(\cdot,x(\cdot)) \in PAPS^{p} (\mathbb{R},Y,\mu) .$ 
\end{theorem}

\noindent {\bf Proof.} 
We have the following decomposition:
\begin{eqnarray}
f(t,u(t))&=& \underbrace{\tilde{f}(t,x_{1}(t))}+ \underbrace{\left[f(t,x(t))-f(t,x_{1}(t)) \right]}+ \underbrace{\varphi(t,x_{1}(t)) }\nonumber \\ &:=& \tilde{F}(t)+F(t)+ \Psi(t),\quad t\in {\mathbb R}.
\end{eqnarray}
Using Theorem \ref{ThmComApSpChapter1}, it follows that $\tilde{F}\in APS^{p} (\mathbb{R},Y)$ and by using Corollary \ref{LemErgSpComChapter1}, we deduce that $ \Psi \in \mathcal{E}^{p} (\mathbb{R},Y,\mu)  .$ Now, it suffices to prove that $ F \in \mathcal{E}^{p} (\mathbb{R},Y,\mu) . $ In view of Lemma \ref{lemmaErgChapter1}, we have 
\begin{eqnarray}
\lim_{r\rightarrow +\infty} \dfrac{\mu \left(M_{\varepsilon, r} (x_{2})  \right) }{\mu([-r,r])} =0,\quad \varepsilon  > 0. \label{ComResErForX2Chapter1} 
\end{eqnarray} 
Let $ \varepsilon  > 0 $. Then, for $ r >0 $ large enough, we have
\begin{eqnarray}
 & &\frac{1}{\mu(\left[-r,r \right])}\int_{-r}^{r}\left( \int_{t}^{t+1}\| F(s)\|_{Y}^{p}ds\right) ^{\frac{1}{p}}d\mu(t) \nonumber  \\ & \leq &\frac{1}{\mu(\left[-r,r \right])}\int_{M_{\varepsilon, r} (x_{2})}\left( \int_{t}^{t+1}\| F(s)\|_{Y}^{p}ds\right) ^{\frac{1}{p}}d\mu(t)  \nonumber \\ &+&\frac{1}{\mu(\left[-r,r \right])}\int_{\left[-r,r \right] \setminus M_{\varepsilon, r} (x_{2}) }\left( \int_{t}^{t+1}\| F(s)\|_{Y}^{p}ds\right) ^{\frac{1}{p}}d\mu(t) \nonumber \\ & \leq & \| F\|_{BS^{p}} \dfrac{\mu \left(M_{\varepsilon, r} (x_{2})  \right) }{\mu([-r,r])} \nonumber \\ &+&\frac{1}{\mu(\left[-r,r \right])} \int_{\left[-r,r \right] \setminus M_{\varepsilon, r} (x_{2}) }\left( \int_{t}^{t+1}\| f(s,x(s))-f(s,x_{1}(s)) \|^{p}ds\right) ^{\frac{1}{p}}d\mu(t). \nonumber \\ \label{ThmComSpForChapter1}
\end{eqnarray}
Let $ K :=\overline{ \lbrace x_{1}(t) : t \in \mathbb{R}\rbrace}$. From $x_{1} \in AP(\mathbb{R},X) $, we assert that $K$ is a compact subset of $X$. Define 
$$ \Phi: X\longrightarrow PAPS^{p}(\mathbb{R},Y) \, \mbox{ through }\; x\mapsto f(\cdot,x) .$$
Since $f \in PAPS^{p}U(\mathbb{R}\times X,Y,\mu)$, using Proposition \ref{PropApSpCom1Chapter1} we may deduce that the restriction of $\Phi $ on any compact $K$ of $X$ is uniformly continuous, which is equivalent to saying that the function $\Phi $ is continuous on $X$. If we apply Lemma \ref{LemSchChapter1} on $\Phi $, we get that, for every $\varepsilon >0$, there exists $\delta >0$ such that, for every $t \in \mathbb{R}$ and $\xi_{1} ,\ \xi_{2} \in X$, we have
\begin{eqnarray*}
\xi_{1} \in K \quad \text{and} \quad \| \xi_{1} - \xi_{2}\| \leq \delta \Rightarrow \left( \int_{t}^{t+1}\|f(s,\xi_{1})-f(s,\xi_{2}) \|_{Y}^{p}ds\right) ^{\frac{1}{p}}\leq \varepsilon .
\end{eqnarray*}
Since $x(t) = x_{1}(t) + x_{2}(t)$ and $x_{1}(t) \in  K$, we have
\begin{align*}
t \in \mathbb{R}& \quad \text{and} \quad \| x_{2}(s)\| \leq \delta  \quad \text{for} \; s\in [t,t+1] \\ & \Rightarrow \left( \int_{t}^{t+1}\|f(s,x(s))-f(s,x_{1}(s)) \|_{Y}^{p}ds\right) ^{\frac{1}{p}} \leq \varepsilon .
\end{align*}
Therefore, by the fact that $ x_{2} \in \mathcal{E} (\mathbb{R},X,\mu) ,$ we have $$\limsup_{r\rightarrow +\infty} \dfrac{\mu \left(M_{\delta, r} (x_{2})  \right) }{\mu([-r,r])}=0.$$  
Using \eqref{ThmComSpForChapter1}, we obtain 
\begin{eqnarray*}
\limsup_{r\rightarrow +\infty}\frac{1}{\mu(\left[-r,r \right])}\int_{-r}^{r}\left( \int_{t}^{t+1}\| F(s)\|_{Y}^{p}ds\right) ^{\frac{1}{p}}d\mu(t)  \leq   \varepsilon \quad \text{for all}\;  \varepsilon >0.  
\end{eqnarray*}
Consequently, 
\begin{eqnarray*}
\lim_{r\rightarrow +\infty}\frac{1}{\mu(\left[-r,r \right])}\int_{-r}^{r}\left( \int_{t}^{t+1}\| F(s)\|_{Y}^{p}ds\right) ^{\frac{1}{p}}d\mu(t)  = 0 .
\end{eqnarray*}

\begin{remark}\label{CompResMuRem}
Notice that condition (iii) is only needed to prove that $  f(\cdot,u(\cdot)) \in BS^{p}(\mathbb{R},Y)$ for all $ 1\leq p <\infty $. Further, the condition (iii) is satisfied for a wide large class of functions $f:\mathbb{R}\times X \longrightarrow Y,$ for more details, see Remark \ref{remark P3} and  Lemma \ref{Lemma Bounded Sp} below. 
\end{remark}
Keeping in mind Theorem \ref{CompResMuChapter1} and Remark \ref{CompResMuRem}, we obtain the following Corollary.
\begin{corollary}\label{CompResMuCoro}
Let $ 1\leq p < +\infty $ and $ \mu \in \mathcal{M} $. Assume that $f:\mathbb{R}\times X \longrightarrow Y$ satisfies the following:
\begin{itemize}
\item[(i)] $ f=\tilde{f}+\varphi \in PAPS^{p}U(\mathbb{R}\times X,Y,\mu)$ with $ \tilde{f}\in APS^{p}U(\mathbb{R}\times X,Y) $ and $ \varphi \in \mathcal{E}^{p}U(\mathbb{R}\times X,Y,\mu) $; 
\item[(ii)] $x=x_{1}+x_{2} \in PAP(\mathbb{R},X,\mu) ,$ where $ x_{1} \in AP(\mathbb{R},X) $ and $x_{2} \in \mathcal{E}(\mathbb{R},X,\mu)$;
\item[(iii)] There exists a non-negative scalar function $L (\cdot) \in BS^{p}(\mathbb{R})$ such that 
$$ \| f(t,x)-f(t,y)\|_{Y} \leq L (t)\|x-y \|, \quad x,\ y\in X,\ \, t\in \mathbb{R}.$$
\end{itemize}
Then, $ f(\cdot,x(\cdot)) \in PAPS^{p} (\mathbb{R},Y,\mu) .$ 
\end{corollary}

\section{Applications to the abstract Volterra integro-differential inclusions and nonautonomous semilinear evolution equations }\label{Section4}

In this section, we shall apply our theoretical results proved so far in the qualitative analysis of bounded solutions for various kinds of abstract semilinear evolution equations in Banach spaces with applications to reaction-diffusion equations. 

\subsection{Pseudo almost periodic solutions of the fractional semilinear inclusions}\label{Section41}
Consider the following semilinear evolution inclusion:
\begin{align*}
D_{t,+}^{\gamma}u(t)\in {\mathcal A}u(t)+f(t,u(t)),\ t\in {\mathbb R},
\end{align*}
where $D_{t,+}^{\gamma}$ denotes the Riemann-Liouville fractional derivative of order $\gamma \in (0,1),$ 
$f : {\mathbb R} \times X \rightarrow X$ is Stepanov $ \mu  $-pseudo almost periodic in $t \in \mathbb{R}$ and satisfies certain properties with respect to $x \in X$. We assume that the multivalued linear operator ${\mathcal A}$ satisfies  the following condition:\\
\textbf{(P1)} There exist $c, \ M>0 $ and $\beta \in (0,1]$ such that 
\begin{equation*}
\Sigma:= \lbrace \lambda \in \mathbb{C} : Re(\lambda) \geq -c (|Im( \lambda )|+1)  \rbrace \subset \rho (\mathcal{A})
\end{equation*}
and 
\begin{equation*}
\|R(\lambda, \mathcal{A}) \| \leq \dfrac{M}{(1+|\lambda |)^{\beta}} \quad  \text{   for all } \lambda \in \Sigma.
\end{equation*}
It is clear that any sectorial operator in the sense of \cite[Definition 4.1]{EN} satisfies \textbf{(P1)}, see \cite[Remark 1 of Proposition 3.6]{Favini}; for more details about this concept and more further results see also the monograph \cite{nova-mono}. Let $(R_{\gamma}(t))_{t>0}$ be the operator family considered in \cite{nova-mono}. Then we know that
\begin{align}\label{deran98765}
\|R_{\gamma}(t)\|=O\bigl(t^{\gamma-1}\bigr),\quad t \in (0,1]\mbox{  and  }\|R_{\gamma}(t)\|=O\bigl(t^{-\gamma-1}\bigr),\quad t \geq 1.
\end{align}
It is said that a continuous function $u: {\mathbb R} \rightarrow X$ is a mild solution of \eqref{mqwert12345} if and only if
\begin{align}
u(t)=\int^{t}_{-\infty}R_{\gamma}(t-s)f\bigl(s,u(s)\bigr)\, ds,\quad t\in {\mathbb R}. \label{integral Sol inclusion}
\end{align}


In the subsequent application of Theorem \ref{CompResMuChapter1}, we will use the following result of A. Meir and E. Keeler \cite{kiler}, who employed the so-called condition of weakly uniformly strict contraction:
\begin{lemma}\label{kiler-manu}
Suppose that $(X,d)$ is a complete metric space and $T : X\rightarrow X$ satisfies that for each $\varepsilon>0$ there exists $\delta>0$ such that for each $x,\ y\in X$ we have:
$$
\varepsilon\leq d(x,y)\leq \varepsilon +\delta \mbox{  implies  }d(Tx,Ty)<\varepsilon.
$$
Then the mapping $T$ has a unique fixed point $\xi,$ and moreover, $\lim_{n\rightarrow +\infty}T^{n}x=\xi$ for any $x\in X.$
\end{lemma}

Therefore, we provide the following hypotheses on $f$:\\

\textbf{(P2)} There exists $ L\geq 0 $ such that for all $\varepsilon>0$ there exists $\delta>0$ satisfying:
$$
\varepsilon \leq\|x-y\|< \varepsilon +\delta\ \  \mbox{  implies  } \sup_{t\in \mathbb{R}}\left( \int^{t}_{t+1} \|f(s,x)-f(s,y)\|^{p} ds\right)^{\frac{1}{p}}  < L\varepsilon \text{ for all } \; x,y \in X.
$$

\textbf{(P3)} For every bounded subset $B\subseteq X$, we have $$\sup_{x \in B} \sup_{t \in \mathbb{R}}\left( \int_{t}^{t+1}  \|f(s,x)\|^{p} ds \right)^{\frac{1}{p}} < \infty .$$ 
\begin{remark}\label{remark P3}
Notice that any function $f:\mathbb{R}\times X \longrightarrow X$ with $f(\cdot,x_0) \in  BS^{p}(\mathbb{R},X)$ for some $x_0 \in X$, and it is Lipschitzian with respect to the second argument, i.e., there exists a non-negative scalar function $L (\cdot) \in BS^{p}(\mathbb{R})$ (for $ 1\leq p<\infty  $) such that 
$$ \| f(t,x)-f(t,y)\| \leq L (t)\|x-y \|, \quad x,\ y\in X,\ \, t\in \mathbb{R},$$
in particular satisfies the hypothesis \textbf{(P3)}.
\end{remark}
In the next, we show that a function satisfying \textbf{(P2)} is not necessarily a strict contraction, but, using Lemma \ref{kiler-manu} it has a unique fixed point. We introduce two examples, the first is in the scalar case and the second one is in the Banach-valued setting.
\begin{examples}\label{M-K examples}
\begin{itemize}
\item[(1)] Set the Banach space  $X= (\mathbb{R}, | \cdot |)$ and define the scalar-valued function $ g: X \longrightarrow X $ given by 
\begin{equation*}
g(x)= \dfrac{|x|}{1+|x|} 
\end{equation*}
Let $\varepsilon>0 $ and $ \delta >0 $ such that $ \varepsilon\leq | x-y| <\delta +\varepsilon $ for all $ x,y \in \mathbb{R} $.  Then, 
\begin{align}
|g(x)-g(y) | =&\Biggl| \dfrac{ |x| }{1+ |x|}-  \dfrac{ |y|}{1+|y|} \Biggr| \nonumber \\
\leq &  \dfrac{| |x|-|y||  }{ 1+|x|+|y|+|x| \,  |y| }  \nonumber  \\
\leq &  \dfrac{| |x|-|y||  }{ 1+|x|+|y| } \label{Example M_K 1}
\end{align} 
Moreover, by assumptions and \eqref{Example M_K 1}, we have
\begin{align*}
|g(x)-g(y) | < & \dfrac{(\varepsilon+ \delta )}{ 1+\varepsilon } := \dfrac{(\varepsilon+ \varepsilon^{2})}{ 1+\varepsilon}= \varepsilon \; (\text{by choosing } \delta=\varepsilon^2).
\end{align*} 
Thus, \textbf{(P2)} holds with $L=1$ and $f(s,x)=g(x)$ for all $s\in \mathbb{R}$ and from \eqref{Example M_K 1}, we obtain that 
$$ |g(x)-g(y) |  \leq | x-y|   \quad \text{for all } x,y \in \mathbb{R}. $$
If we assume that $g$ is a strict contraction mapping, then there exists $0<\tilde{K}<1$ such that 
$$ |g(x)-g(y) |  \leq \tilde{K} | x-y|   \quad \text{for all } x,y \in \mathbb{R}.  $$
Let $ 0\neq x\in X $ and  $y=0$. Then, $|g(x)-g(0)|=g(x) \leq \tilde{K} |x| $ which implies that 
$$   \dfrac{1}{|x|} \dfrac{|x|}{|x|+1}=\dfrac{1}{|x|+1} \leq \tilde{K}. $$
Therefore, by letting $ x\rightarrow 0 $, we obtain that $1\leq \tilde{K}$ which yields a contradiction. Hence, $g$ is not a strict contraction. 
\item[(2)] Let the Banach space $X:= (L^{2}(\Omega), \| \cdot \|)$ (equipped with its usual norm), where $ \Omega \subset \mathbb{R}^{n} $ is any bounded open set. Define the function $f:\mathbb{R}\times X\longrightarrow X $  by 
$$   f (t,\varphi)(x)=K(t)  \dfrac{\| \varphi \|}{1+ \|\varphi \| }Q(x) +H(t,x), \; t\in \mathbb{R}, \, x\in \Omega, $$
where $K:\mathbb{R}\longrightarrow (0,\infty) $, $ Q \in X $, $ Q \geq 0 $ with $ \|K\|_{BS^p} \| Q\| \geq 1 $ and $ H :\mathbb{R}\times \Omega \longrightarrow [0, \infty) $.
The function $f$ satisfies \textbf{(P2)} but it is not a strict contraction. In fact, let $ \varphi_{1},\varphi_{2} \in X $ and $\varepsilon, \ \delta>0$ such that
$\varepsilon \leq \|\varphi_1-\varphi_2\| < \varepsilon +\delta$. Then, a straightforward calculation yields:
\begin{eqnarray}
| f(t,\varphi_{1})(x)-f(t,\varphi_{2})(x)| &=& K(t) \Biggl| \dfrac{\| \varphi_{1}\|}{1+\| \varphi_{1}\| }-\dfrac{ \| \varphi_{2}\|}{1+\| \varphi_{2}\|}  \Biggr| Q(x) \nonumber \\
 &\leq & K(t)  \dfrac{ \mid \| \varphi_{1}\|-\|\varphi_{2} \|\mid }{ 1+\| \varphi_{1} \|+\| \varphi_{2} \|+\|\varphi_{1} \| \, \| \varphi_{2} \|} Q(x) \nonumber \\ 
&< & K(t)  \dfrac{\delta + \varepsilon  }{ 1+  \varepsilon } Q(x) \nonumber \\
&:= & K(t)  \dfrac{\varepsilon^2 + \varepsilon  }{ 1+  \varepsilon } Q(x) \nonumber \\
&=& K(t) Q(x) \ \varepsilon, 
\; t \in \mathbb{R}, \,x\in \Omega. \label{M-K condition estim 1}
\end{eqnarray} 
Using \eqref{M-K condition estim 1},  we have
\begin{eqnarray*}
\| f(t,\varphi_1)-f(t,\varphi_2)\|^{2} &=& \int_{\Omega}| f(t,\varphi_{1})(x)-f(t,\varphi_{2})(x)|^{2}dx \\
 &< & K(t)^{2}\int_{\Omega}  Q(x)^{2} dx \ \varepsilon^{2} \\ 
& = & K(t)^{2}  \| Q \|^{2} \ \varepsilon^{2}, \quad t \in \mathbb{R}.
\end{eqnarray*} 
Finally, 
\begin{eqnarray*}
\left( \int_{t}^{t+1}\|f(s,\varphi_{1})-f(s,\varphi_{1}) \|^{p} ds\right)^{\frac{1}{p}}  < \| K\|_{BS^p}\| Q \| \ \varepsilon \quad \mbox{for all}  \; t\in \mathbb{R}.
\end{eqnarray*}
This proves the result with $L:=\| K\|_{BS^p} \| Q\|$. To show that $f$ is not necessarily a strict contraction, we assume the converse i.e., there exists $0<\tilde{K}  <1$ such that for all $ \varphi_1 , \varphi_2  \in X$ we have 
$$  \| f(t,\varphi_1)-f(t,\varphi_2)\| \leq \tilde{K} \| \varphi_1 -\varphi_2\|, \quad t\in \mathbb{R}.  $$
Set for each $ n\geq 1 $, $0 \neq \varphi_{n}(x):= e^{-n\sqrt{|x|}}$ for all $ x \in \Omega  $ where $|\cdot|$ denotes a norm in $\mathbb{R}^{n}$. It is clear that $(\varphi_n)_n \subset X$ and $ \lim_{n\rightarrow \infty}\|\varphi_n \| =0$. Therefore,
$$ K(t)\dfrac{\| \varphi_{n}\|}{1+\| \varphi_{n}\| } \|Q \| \leq \tilde{K} \| \varphi_{n} \| \; \text{ for all } n \geq 1, t\in \mathbb{R}, $$
which implies that $$K(t) \dfrac{1 }{1+\| \varphi_{n}\| } \|Q \| \leq \tilde{K} \; \text{ for all } n \geq 1, t\in \mathbb{R}.$$
Now, by letting $ n\rightarrow +\infty $, we obtain that 
$$K(t) \|Q \| \leq \tilde{K} \text{ for all } t\in \mathbb{R}. $$
Thus, if we pass to the $BS^p$-norm, we obtain that 
$$1\leq \| K\|_{BS^p} \|Q \| \leq \tilde{K},$$
which contradicts the assumption on $\tilde{K}$. Hence $f$ is not necessarily a strict contraction.
\end{itemize}
\end{examples}
Now we will state our first main result. We shall assume that $ 1<p<\infty $. 
\begin{proposition}\label{mnemojwer0}
Let $ \mu \in \mathcal{M} $ satisfy \textbf{(M)}. Suppose that $   f(\cdot,u(\cdot)) \in PAPS^{p}(\mathbb{R},X, \mu)$. Then the mapping given by $$(F_0 u )(t):=  \int^{t}_{-\infty}R_{\gamma}(t-s)f(s,u(s))\, ds,\quad  \, t\in \mathbb{R}$$
maps $PAP({\mathbb R},X, \mu)$ into $PAP({\mathbb R},X, \mu)$.
\end{proposition} 
\begin{proof}
For any $ u \in  PAP({\mathbb R},X, \mu)$, we set
$$
(F_0 u)(t):=\int^{t}_{-\infty}R_{\gamma}(t-s)f(s,u(s))\, ds,\quad  \, t\in \mathbb{R}.
$$
Since $ s\longmapsto f(s,u(s)) \in BS^{p}(\mathbb{R},X)$, we obtain that
\begin{align*}
\bigl \| F_0 u(t) \bigr\| 
& \leq \int^{t}_{-\infty}\bigl \| R_{\gamma}(t-s)\bigl\| \cdot \bigl\| f(s,u(s))\bigr\|\, ds
\\& = \int^{\infty}_{0}\bigl \| R_{\gamma}(s)\bigl\| \cdot \bigl\| f(t-s,u(t-s))\bigr\|\, ds
\\& = \int^{1}_{0}\bigl \| R_{\gamma}(s)\bigl\| \cdot \bigl\| f(t-s,u(t-s))\bigr\|\, ds +\int^{\infty}_{1}\bigl \| R_{\gamma}(s)\bigl\| \cdot \bigl\| f(t-s,u(t-s))\bigr\|\, ds
\\& \leq  \left( \int^{1}_{0} s^{q(\gamma -1)}\right)^{\frac{1}{q}} \left( \int_{0}^{1} \bigl\| f(t-s,u(t-s))\bigr\|^{p}\, ds\right)^{\frac{1}{p}} 
\\& +\sum_{k\geq 1}k^{-\gamma-1}\left(  \int^{k+1}_{k}  \bigl\| f(t-s,u(t-s))\bigr\|^{p}\, ds\right)^{\frac{1}{p}} 
\\& =  \left(1-q(\gamma -1) \right)^{\frac{1}{q}} \left( \int_{0}^{1} \bigl\| f(t-s,u(t-s))\bigr\|^{p}\, ds\right)^{\frac{1}{p}} 
\\& +S_{\gamma}\left(  \int^{k+1}_{k}  \bigl\| f(t-s,u(t-s))\bigr\|^{p}\, ds\right)^{\frac{1}{p}} 
\\& \leq  \left(  \left(1-q(\gamma -1) \right)^{\frac{1}{q}}+S_{\gamma}\right)  \| f(\cdot,u(\cdot))\bigr\|_{BS^p}, \quad t\in \mathbb{R}.
\end{align*}
Thus $ F_0 $ is well-defined and  yields a continuous bounded mapping. Furthermore, from the fact that $ s\longmapsto f(s,u(s)) \in PAPS^{p}({\mathbb R} ,X,\mu)$ we obtain by definition $ f(s,u(s))=\tilde{f}(s,u(s))+\varphi(s,u(s)) $, where $ s\longmapsto \tilde{f}(s,u(s)) \in APS^{p}({\mathbb R} ,X)$ and  $ s\longmapsto \varphi (s,u(s)) \in \mathcal{E}^{p}({\mathbb R} ,X,\mu)$. Then, we obtain
$$
(F_0 u)(t):=\int^{t}_{-\infty}R_{\gamma}(t-s)\tilde{f}(s,u(s))\, ds+\int^{t}_{-\infty}R_{\gamma}(t-s)\varphi(s,u(s))\, ds, \quad t\in {\mathbb R}.$$
Let $ \varepsilon>0 $, since $ s\longmapsto \tilde{f}(s,u(s)) \in APS^{p}(\mathbb{R},X)$, there exists $ l_\varepsilon >0 $ such that each interval of length $ l_\varepsilon $ contains an element $ \tau $ such that 
$$   \left( \int_{t}^{t+1}\|\tilde{f}(s+\tau,u(s+\tau))-\tilde{f}(s,u(s))\|^{p} ds\right)^{\frac{1}{p}}  <\varepsilon /(\left(1-q(\gamma -1) \right)^{\frac{1}{q}}+ S_{\gamma} )$$
uniformly in $ t\in  \mathbb{R} $. Hence, by H\"older inequality, we have
\begin{align*}
& \bigl \| F_0  u(t+\tau)-F_0  u(t) \bigr\|  \\
& \leq \int^{\infty}_{0}\bigl \| R_{\gamma}(s)\bigl\| \cdot \bigl\| \tilde{f}(t+\tau-s,u(t+\tau-s))-\tilde{f}(t-s,u(t-s))\bigr\|\, ds
\end{align*}
\begin{align*}
\\& = \int^{1}_{0}\bigl \| R_{\gamma}(s)\bigl\| \cdot \bigl\| \tilde{f}(t+\tau-s,u(t+\tau-s))-\tilde{f}(t-s,u(t-s))\bigr\|\, ds
\\& +\int^{\infty}_{1}\bigl \| R_{\gamma}(s)\bigl\| \cdot \bigl\| \tilde{f}(t+\tau-s,u(t+\tau-s))-\tilde{f}(t-s,u(t-s))\bigr\|\, ds
\\& \leq  \left( \int^{1}_{0} s^{q(\gamma -1)}\right)^{\frac{1}{q}} \left( \int_{0}^{1} \bigl\| \tilde{f}(t+\tau-s,u(t+\tau-s))-\tilde{f}(t-s,u(t-s))\bigr\|^{p}\, ds\right)^{\frac{1}{p}} 
\\& +\sum_{k\geq 1}k^{-\gamma-1}\left(  \int^{k+1}_{k}  \bigl\| \tilde{f}(t+\tau-s,u(t+\tau-s))-\tilde{f}(t-s,u(t-s))\bigr\|^{p}\, ds\right)^{\frac{1}{p}} 
\\& =  \left(1-q(\gamma -1) \right)^{\frac{1}{q}} \left( \int_{0}^{1} \bigl\| \tilde{f}(t+\tau-s,u(t+\tau-s))-\tilde{f}(t-s,u(t-s))\bigr\|^{p}\, ds\right)^{\frac{1}{p}} 
\\& +S_{\gamma}\left(  \int^{k+1}_{k}  \bigl\| \tilde{f}(t+\tau-s,u(t+\tau-s))-\tilde{f}(t-s,u(t-s))\bigr\|^{p}\, ds\right)^{\frac{1}{p}}
 \\& \leq \varepsilon, \quad t \in \mathbb{R}. 
\end{align*}
Therefore, it suffices to prove that $t\longmapsto \int^{t}_{-\infty}R_{\gamma}(t-s)\varphi(s,u(s))\, ds \in \mathcal{E}({\mathbb R} ,X,\mu)$. Indeed, let $ r>0 $. Then, by the H\"older inequality we obtain
\begin{align*}
&\dfrac{1}{\mu([-r,r])}\int_{-r}^{r}\bigl \| F_0 u(t) \bigr\| d\mu(t) \\
& \leq\dfrac{1}{\mu([-r,r])}\int_{-r}^{r} \int^{t}_{-\infty}\bigl \| R_{\gamma}(t-s)\bigl\| \cdot \bigl\| f(s,u(s))\bigr\|\, ds d\mu(t)
\\& = \dfrac{1}{\mu([-r,r])}\int_{-r}^{r}\ \int^{\infty}_{0}\bigl \| R_{\gamma}(s)\bigl\| \cdot \bigl\| f(t-s,u(t-s))\bigr\|\, ds d\mu(t)
\\& =\dfrac{1}{\mu([-r,r])}\int_{-r}^{r} \int^{1}_{0}\bigl \| R_{\gamma}(s)\bigl\| \cdot \bigl\| f(t-s,u(t-s))\bigr\|\, ds d\mu(t)
\\&+\dfrac{1}{\mu([-r,r])}\int_{-r}^{r} \int^{\infty}_{1}\bigl \| R_{\gamma}(s)\bigl\| \cdot \bigl\| f(t-s,u(t-s))\bigr\|\, ds d\mu(t)
\\& \leq \dfrac{\left(1-q(\gamma -1) \right)^{\frac{1}{q}}}{\mu([-r,r])}\int_{-r}^{r}  \left( \int_{0}^{ 1} \bigl\| f(t-s,u(t-s))\bigr\|^{p}\, ds\right)^{\frac{1}{p}} d\mu(t)
\\& +\dfrac{S_{\gamma}}{\mu([-r,r])}\int_{-r}^{r} \left(  \int^{k+1}_{k}  \bigl\| f(t-s,u(t-s))\bigr\|^{p}\, ds\right)^{\frac{1}{p}} d\mu(t)
\\& \rightarrow 0 \text{ as } r\rightarrow +\infty.
\end{align*}
We recall from Proposition \ref{Trans inva Ep} (i), that the set $ \mathcal{E}^{p}({\mathbb R} ,X,\mu) $ is translation invariant. 
Hence, the result follows immediately.
\end{proof}
\begin{theorem}\label{mnemojwer2}
Let $ 1<p < +\infty $ and $ \mu \in \mathcal{M} $ satisfy \textbf{(M)}. Assume that $f\in PAPS^{p}U({\mathbb R} \times X,X,\mu)$ such that \textbf{(P2)} and \textbf{(P3)} hold with $  L S_{\gamma}^{q} \leq 1$ where $$S^{q}_{\gamma}:=\sum_{k\geq 1}k^{-\gamma-1}+ \left(1-q(\gamma -1) \right)^{\frac{1}{q}}.$$ Then, the inclusion \eqref{mqwert12345} has a unique $ \mu $-pseudo almost periodic mild solution given by the integral representation \eqref{integral Sol inclusion}.
\end{theorem} 
\begin{proof}
Let $ u \in PAP({\mathbb R} ,X,\mu)$. From \textbf{(P3)}, we have that the function $ s\longmapsto f(s,u(s))$ belongs to the space $PAPS^{p}({\mathbb R} ,X,\mu)$ in view of Theorem \ref{ThmComApSpChapter1}. Then, by Proposition \ref{mnemojwer0}, $ F_0  $ maps $ PAP({\mathbb R} ,X,\mu) $ into itself. Therefore, it suffices to prove that $ F_0 $ has a unique fixed point in $ PAP({\mathbb R},X,\mu) $ using Lemma \ref{kiler-manu}. Let $ \varepsilon>0 $ and $\delta>0$ be determined from\textbf{(P2)}. For all $u, v \in PAP({\mathbb R},X,\mu)$, $\varepsilon \leq \| u(t)-v(t) \| < \varepsilon +\delta $ for all $ t\in \mathbb{R} $. Then, the hypothesis \textbf{(P2)} yields:
\begin{align*}
\bigl \| F_0 u(t)-F_0 v(t) \bigr\| & \leq \int^{t}_{-\infty}\bigl \| R_{\gamma}(t-s)\bigl\| \cdot \bigl\| f(s,u(s))-f(s,v(s))\bigr\|\, ds
\\& = \int^{\infty}_{0}\bigl \| R_{\gamma}(s)\bigl\| \cdot \bigl\| f(t-s,u(t-s))-f(t-s,u(t-s))\bigr\|\, ds
\\& = \int^{1}_{0}\bigl \| R_{\gamma}(s)\bigl\| \cdot \bigl\| f(t-s,u(t-s))-f(t-s,v(t-s))\bigr\|\, ds
\\& +\int^{\infty}_{1}\bigl \| R_{\gamma}(s)\bigl\| \cdot \bigl\| f(t-s,u(t-s))-f(t-s,v(t-s))\bigr\|\, ds
\\& \leq  \left( \int^{1}_{0} s^{q(\gamma -1)}\right)^{\frac{1}{q}} \left( \int_{0}^{1} \bigl\| f(t-s,u(t-s))-f(t-s,v(t-s))\bigr\|^{p}\, ds\right)^{\frac{1}{p}} 
\\& +\sum_{k\geq 1}k^{-\gamma-1}\left(  \int^{k+1}_{k}  \bigl\| f(t-s,u(t-s))-f(t-s,v(t-s))\bigr\|^{p}\, ds\right)^{\frac{1}{p}} 
\\& =  \left(1-q(\gamma -1) \right)^{\frac{1}{q}} \left( \int_{0}^{1} \bigl\| f(t-s,u(t-s))-f(t-s,v(t-s))\bigr\|^{p}\, ds\right)^{\frac{1}{p}} 
\\& +\sum_{k\geq 1}k^{-\gamma-1}\left(  \int^{k+1}_{k}  \bigl\| f(t-s,u(t-s))-f(t-s,v(t-s))\bigr\|^{p}\, ds\right)^{\frac{1}{p}}
\\& < S^{q}_{\gamma}L \varepsilon \quad \text{ for all } t \in \mathbb{R}. 
\end{align*}
Hence, by the assumption $ S^{q}_{\gamma}L  \leq 1 $ we obtain that 
\begin{align*}
\bigl \| F_0 u-F_0 v \bigr\|_{\infty} < \varepsilon .
\end{align*}
The result follows immediately from Lemma \ref{kiler-manu}.
\end{proof}
In order to visualize the advantage of our result Theorem \ref{mnemojwer2}, in the next Theorem, we will prove the existence and uniqueness of $ \mu $-pseudo almost periodic solutions to the inclusion \eqref{mqwert12345} under the following Lipschitz assumption:\\ 

\textbf{(Q)}  There exists a non-negative scalar function $L (\cdot) \in BS^{p}(\mathbb{R})$ (for $ 1 < p<\infty  $) such that 
$$ \| f(t,x)-f(t,y)\| \leq L (t)\|x-y \|, \quad x,\ y\in X,\ \, t\in \mathbb{R} .$$
\begin{theorem}\label{Coro M-K}
Let $ 1<p < +\infty $ and $ \mu \in \mathcal{M} $ satisfy \textbf{(M)}. Assume that $f\in PAPS^{p}U({\mathbb R} \times X,X,\mu)$ such that \textbf{(Q)} holds with $  \| L \|_{BS^{p}} S_{\gamma}^{q} <1$. Then the inclusion \eqref{mqwert12345} has a unique $ \mu $-pseudo almost periodic mild solution given by the integral representation \eqref{integral Sol inclusion}.
\end{theorem}
\begin{proof}
From Theorem \ref{mnemojwer2} and Corollary \ref{CompResMuCoro}, it suffices to prove that the mapping $F_0$ has a unique fixed point. Indeed, let $u, v \in PAP({\mathbb R},X,\mu)$. Then, by \textbf{(Q)}, we get
\begin{align*}
\bigl \| F_0 u(t)-F_0 v(t) \bigr\| & \leq \int^{t}_{-\infty}\bigl \| R_{\gamma}(t-s)\bigl\| \cdot \bigl\| f(s,u(s))-f(s,v(s))\bigr\|\, ds
\\& = \int^{\infty}_{0}\bigl \| R_{\gamma}(s)\bigl\| \cdot \bigl\| f(t-s,u(t-s))-f(t-s,u(t-s))\bigr\|\, ds
\\& = \int^{1}_{0}\bigl \| R_{\gamma}(s)\bigl\| \cdot \bigl\| f(t-s,u(t-s))-f(t-s,v(t-s))\bigr\|\, ds
\\& +\int^{\infty}_{1}\bigl \| R_{\gamma}(s)\bigl\| \cdot \bigl\| f(t-s,u(t-s))-f(t-s,v(t-s))\bigr\|\, ds
\\& \leq  \left( \int^{1}_{0} s^{q(\gamma -1)}\right)^{\frac{1}{q}} \left( \int_{0}^{1} \bigl\| f(t-s,u(t-s))-f(t-s,v(t-s))\bigr\|^{p}\, ds\right)^{\frac{1}{p}} 
\\& +\sum_{k\geq 1}k^{-\gamma-1}\left(  \int^{k+1}_{k}  \bigl\| f(t-s,u(t-s))-f(t-s,v(t-s))\bigr\|^{p}\, ds\right)^{\frac{1}{p}} 
\\& \leq   \left(1-q(\gamma -1) \right)^{\frac{1}{q}} \left( \int_{0}^{1}  L(t-s)^{p}\, ds\right)^{\frac{1}{p}} \| u-v\bigr\|_{\infty}
\\& +\sum_{k\geq 1}k^{-\gamma-1}\left( \int_{k}^{k+1}  L(t-s)^{p}\, ds\right)^{\frac{1}{p}} \| u-v\bigr\|_{\infty}
\\& \leq S^{q}_{\gamma}\|L \|_{BS^p}  \| u-v\bigr\|_{\infty}, \quad t \in \mathbb{R}. 
\end{align*}
Hence, we obtain that 
\begin{align*}
\bigl \| F_0 u-F_0 v \bigr\|_{\infty} \leq S^{q}_{\gamma}\|L \|_{BS^p}  \| u-v\bigr\|_{\infty} .
\end{align*}
Then, the result follows from the Banach strict contraction principle, since by assumption $ S^{q}_{\gamma}\|L \|_{BS^p} <1 $.
\end{proof}
\begin{remark}
It is very important to state that, in Theorem \ref{mnemojwer2}, under the assumptions \textbf{(P2)} and \textbf{(P3)}, the condition  $\| L \|_{BS^{p}} S_{\gamma}^{q}=1$ (i.e. $\| L \|_{BS^{p}}=(S_{\gamma}^{q})^{-1}$) yields the existence and uniqueness of $ \mu $-pseudo almost periodic mild solutions to the inclusion \eqref{mqwert12345} in view of  Meir and Keeler fixed Theorem. However, the existence and uniqueness result does not hold in this case when $f$ satisfies \textbf{(Q)}, see Theorem \ref{Coro M-K}.\\
\end{remark}
\subsubsection{\textbf{Example}} \label{prijem}
Consider the following semilinear  fractional Poisson heat equation in the $ L^2 $-setting, namely 
\begin{equation}
\left\{
    \begin{aligned}
  D_{t}^{\gamma}( m(x)v (t,x))& = \Delta v(t,x) +g(t,v(t,x))+H(t,x) , \  t \in \mathbb{R},\ x\in \Omega, \\   
     v(t,x)&|_{\partial \Omega} = 0 ; \ t \in \mathbb{R},\ x\in \partial\Omega , \label{model App1}
    \end{aligned}
  \right. 
\end{equation}
where $ \gamma \in  (0,1) $, $\Omega \subset \mathbb{R}^{N}$ an open bounded subset with sufficiently smooth boundary $\partial \Omega$ and $ m \in L^{\infty}(\Omega) $, $m\geq 0$. Here $ H: \mathbb{R}\times \Omega \longrightarrow \mathbb{R} $ is $ S^{2} $-$ \mu $-pseudo almost periodic function. 
Let $X= L^{2}(\Omega)$ be the Lebesgue space of square integrable functions on $\Omega$ equipped with its usual norm denoted by $\| \cdot \|$. Define the operator $\mathcal{A}$ on $X$ by $$ \mathcal{A} \varphi:=\Delta M^{-1}\varphi,$$
with $\varphi \in D(\mathcal{A})$ (the maximal domain of $\mathcal{A}$) where $\Delta$ is the Dirichlet Laplacian on $X$ and $M$ is the multiplication operator by $m$ on $X$. In addition, we set $f:\mathbb{R}\times X\longrightarrow X$ given by $$   f(t,\varphi)(x)=g(t,\varphi(x))+H(t,x), \; t\in \mathbb{R}, x\in \Omega .$$
Then, our model \eqref{model App1} is equivalent to the inclusion \eqref{mqwert12345}. It is clear that $\mathcal{A}$ satisfy \textbf{(P1)}, see \cite{nova-mono} and references therein. Moreover, for $\varphi\in X $, we define
$$g(t,\varphi(x)):= K(t)  \dfrac{R(x)}{1+ ||\varphi || } , \; t\in \mathbb{R}, x\in \Omega ,$$
where $ K: \mathbb{R} \longrightarrow (0, \infty)$ is $ S^{2} $-$ \mu $-pseudo almost periodic and $ R \in X $, $ R \geq 0 $. Then, we have
\begin{lemma}\label{Lemma 1 App1}
The function $f$ satisfies \textbf{(P2)} with $L=\| K\|_{BS^2} \|R\|$.
\end{lemma}

\begin{proof}
Let $ \varphi_{1},\varphi_{2} \in X $ and let $\varepsilon>0$ and $\delta>0$ such that
$
\varepsilon \leq  \|\varphi_1-\varphi_2\|< \varepsilon +\delta$. So, 
\begin{eqnarray}
\Biggl | f(t,\varphi_{1})(x)-f(t,\varphi_{2})(x) \Biggr | &=& K(t) \Biggl | \dfrac{ R(x)}{1+\| \varphi_{1} \| }-\dfrac{ R(x)}{1+ \|\varphi_{2}\|} \Biggr | \nonumber \\
 &\leq & K(t) \dfrac{\mid \| \varphi_{1} \| -\| \varphi_{2}\| \mid }{ 1+\| \varphi_{1}\| +\| \varphi_{2}\|} R(x),  \; t\in \mathbb{R},\, x\in \Omega. \label{M-K condition estim 3}
\end{eqnarray} 
Then, using \eqref{M-K condition estim 3}, we get
\begin{eqnarray*}
\int_{\Omega}| f(t,\varphi_{1})(x)-f(t,\varphi_{2})(x)|^{2}dx 
&\leq  & K(t)^{2}  \dfrac{\|\varphi_{1}- \varphi_{2}\| ^{2} }{ (1+\varepsilon )^{2}}\int_{\Omega} R(x)^{2} dx \\
&\leq  & K(t)^{2}  \dfrac{(\delta +\varepsilon)^{2} }{ (1+\varepsilon )^{2}} \| R\| ^{2}  \\
&\leq  &  K(t)^{2} \| R\| ^{2}  \varepsilon^2 , \quad t\in \mathbb{R} \; \text{ (by taking } \delta:=\varepsilon^2 \text{)}.
\end{eqnarray*} 
Hence, for all $ \varphi_{1},\varphi_{2} \in X $, we have
\begin{eqnarray*}
\left( \int_{t}^{t+1}\|f(s,\varphi_{1})-f(s,\varphi_{2}) \|^{2} ds\right)^{\frac{1}{2}}  <\| K \|_{BS^2} \| R\| \, \varepsilon \quad \mbox{for all}  \; t\in \mathbb{R}.
\end{eqnarray*}
This proves the result with $L:=\| K \|_{BS^2} \| R\|$.
\end{proof}
\begin{lemma}
The function $f$ is Lipschitzian with respect to the second argument with Lipschitz constant $L(\cdot):=K(\cdot) \| R\|$. Moreover, $f$ satisfy \textbf{(P3)}.
\end{lemma}
\begin{proof}
Let $ \varphi_{1},\varphi_{2} \in X $. By the proof of Lemma \ref{Lemma 1 App1} we assert that 
$$  \|f(t,\varphi_{1})-f(t,\varphi_{2}) \| \leq  K(t) \| R\| \ \|\varphi_{1}-\varphi_{2} \|, \quad t\in \mathbb{R} .$$
Hence, the result follows from Remark \ref{remark P3}.
\end{proof}
Furthermore, we take $$\|K\|_{BS^2} \| R \|=(S_{\gamma}^{2})^{-1}.$$ 
Then, we have the following main result.
\begin{theorem}
The fractional Poisson heat equation \eqref{model App1} has a unique $ \mu $-pseudo almost periodic intergal solution.
\end{theorem}
\begin{proof}
The result follows from Theorem  \ref{mnemojwer2}.
\end{proof}

%
%

\subsection{Pseudo almost periodic solutions of the equation \eqref{Eq_1}}\label{Section42}
In this subsection, we consider the existence and uniqueness of $\mu$-pseudo almost periodic solutions of the following semilinear nonautonomous evolution equations:
\begin{equation*}
     x'(t)= A(t)x(t) + f(t, x(t))\quad \text{for}\; t \in \mathbb{R}.   
\end{equation*}
Let $(A(t),D(A(t))), \; t\in \mathbb{R}$ be a family of linear closed  operators on a Banach space $X.$ Of concern is the following linear Cauchy problem \begin{align}
 \left\{
    \begin{array}{ll}
        u'(t)=A(t)u(t), & t\geq s, \\
       u(s)=x \in X. & 
    \end{array}  \label{Eq Non Aut Inho}
\right.  
\end{align}
Here, we assume that $(A(t),D(A(t))), \; t\in \mathbb{R}$ generates an evolution family, which solves the problem \eqref{Eq Non Aut Inho},  i.e., a two-parameter family $(U(t,s))_{t\geq s}$ of linear bounded operators in $X$ such that the map $ (t,s)\longmapsto U(t,s) \in \mathcal{L}(X) $ is strongly continuous, $ U(t,s)U(s,r)=U(t,r) \quad \mbox{and}\quad U(t,t)=I \quad \mbox{for} \quad t\geq s\geq r$. A (mild) solution to problem \eqref{Eq Non Aut Inho} is $u(t)=U(t,s)x$ for $t\geq s$. In particular, if $A(t)$ is time-independent, i.e., $A(t)=A$ for all $t \in \mathbb{R}$, then $U(t,s)=T(t-s),$ where $(T(t))_{t\geq 0}$ is a  semigroup of bounded linear operators on $X$. Notice that, in general, the domains $ D(A(t)) $ of the operators $ A(t) $ are not necessarily dense in $X$ and may change with respect to $t$. Unlike semigroups, there is no necessary and sufficient spectral criteria for $(A(t),D(A(t))), \; t\in \mathbb{R}$ to generate an evolution family. In the parabolic context, we refer to \cite{AquTer1, Kato,Ya1} and references therein for some results obtained so far in this direction. If $A(t) $ has a constant domain $D (A (t)) = D$, $ t\in \mathbb{R} ,$ then we have the following generation result:\\

\textbf{(C1)} Let $(A(t),D), \; t\in \mathbb{R}$ be the generators of analytic semigroups $  (T^{t}(\tau))_{\tau\geq 0} $ on $X$ of the same type $(N,\omega);$ that is, $ \| T_{t}(s)\| \leq Ne^{\omega s}$ (uniformly in $t$). Assume that $A(t) $  is invertible for all $ t\in \mathbb{R} $, $\sup_{t,s\in \mathbb{R} } \|A(t)A(s)^{-1} \|< \infty $  and there exist constants $\omega \in \mathbb{R} $, $L \geq 0$ and $0 < \mu \leq 1$ such that
\begin{eqnarray}
\| (A(t)-A(s))R(\omega ,A(r)) \| \leq L |t-s | ^{\mu} \quad \text{for } t,s,r \in \mathbb{R}. \label{AquTerCon21}
\end{eqnarray}

In this case, the map $ (t,s)\longmapsto U(t,s) \in \mathcal{L}(X) $ is continuously differentiable for $ t>s $ with respect to the variable $t$, $ U(t,s) $  maps $ X $ into $ D(A(t)) $ and we have $ \partial U(t,s)/\partial t=A(t)U(t,s) .$ Moreover, $ U(t,s) $ and $ (t-s)A(t)U(t,s) $ are exponentially bounded.

Now we recall the notion of exponential dichotomy of an evolution family (for more details, see \cite{EN,Schnaubelt}):
\begin{definition}
An evolution family $ (U(t,s))_{t\geq s} $ on $X$ is said to have an exponential dichotomy in $ \mathbb{R} $ if there exists a family of projections $P(t) \in \mathcal{L}(X)$, $t \in \mathbb{R}$, being strongly continuous with respect to $t$, and constants  $\delta ,\ N > 0$ such that:
\begin{itemize}
\item[\textbf{(i)}] $ U(t,s) P(s)=P(t)U(t,s); $
\item[\textbf{(ii)}] $U(t,s): Q(s)X\longrightarrow Q(t)X$ is invertible with the inverse $ \tilde{U}(t,s) $ (i.e., $ \tilde{U}(t,s)=U(s,t)^{-1} $);
\item[\textbf{(iii)}] $\|U(t,s)P(s)\| \leq N e^{-\delta (t-s)}$ and $\| \tilde{U}(s,t)Q(t)\| \leq N e^{-\delta (t-s)}$
\end{itemize} 
for all $t,\ s \in \mathbb{R}$ with $t \geq s$, where $Q(t) := I-P(t) .$ If this is the case, then we also say that the evolution family $ (U(t,s))_{t\geq s} $ is hyperbolic.
\end{definition}

Given a hyperbolic evolution family $(U(t,s))_{t\geq s} $, then its associated Green function is 
defined by
\begin{equation}
G(t,s):=\left\{
    \begin{aligned}
     U(t,s) P(s),&  \quad t,s \in \mathbb{R},\; s \leq t,\\   
     -\tilde{U}(t,s) Q(s), &  \quad t,s \in \mathbb{R},\; s > t .\\
    \end{aligned} \label{GreenFun}
  \right. 
\end{equation}
The exponential dichotomy can be characterized in many cases; for more details, see \cite{He}.
From \cite{Schnaubelt}, the exponential dichotomy holds in the following case: \\

\textbf{(C2)} Assume that \textbf{(C1)} holds and the semigroups $(T^{t}(\tau))_{\tau \geq 0}$ are hyperbolic with projections $P_t$ and constants $N, \delta >0$ such that 
$ \| A(t) T^{t}(\tau)P_t \| \leq \psi (\tau) $ and  $ \| A(t)T_{Q}^{t}(\tau)Q_t\| \leq \psi(-\tau) $  for  $ \tau >0$  and a function $ \psi$ such that the mapping $\mathbb{R} \ni s \longmapsto \varphi(s):= |s |^{\mu} \psi (s)$  is integrable with  $ L \| \varphi \|_{L^{1}(\mathbb{R})} <1.$\\
Now, we give our main hypotheses:\\
\textbf{(H1)} The operators  $A(t) $, $ t \in \mathbb{R} $ generates a strongly continuous evolution family $(U(t,s))_{t\geq s}$ on $X$.  \\
\textbf{(H2)} The evolution family $(U(t,s))_{t\geq s}$ has an exponential dichotomy on $\mathbb{R}$ with constants $N,\ \delta > 0$, projections $P(t),\ t \in \mathbb{R}$ and Green's function $G(\cdot,\cdot)$.\\
\textbf{(H3)} $R(\omega,A(\cdot))$ is almost periodic for some $\omega \in \mathbb{R}$.\\
\textbf{(H4)} $f$ is Lipschitzian in bounded sets with respect to the second argument, i.e., for each $\rho> 0$ there exists a non-negative scalar function $L_\rho (\cdot) \in BS^{p}(\mathbb{R})$ (for $ 1\leq p<\infty  $) such that 
$$ \| f(t,x)-f(t,y)\| \leq L_\rho (t)\|x-y \|, \quad x,\ y\in B(0,\rho),\ \, t\in \mathbb{R} .$$

\begin{remark}
\textbf{(a)} Notice that if \textbf{(C2)} is satisfied, then the hypotheses \textbf{(H1)}-\textbf{(H2)} hold.\\
\textbf{(b)} In \cite{Schnaubelt}, the authors have proved that, if $R(\omega,A(\cdot))$ is almost periodic for some $\omega \in \mathbb{R},$ then the associated Green function is bi-almost periodic.
\end{remark}

By a mild solution of \eqref{Eq_1} we mean any continuous function $ u: \mathbb{R} \longrightarrow X $ which satisfies the following variation of constants formula:
\begin{eqnarray}
u(t)=U(t,\sigma)u(\sigma)+ \int_{\sigma}^{t} U(t,s)f(s,u(s))\, ds \qquad \mbox{for all} \; t\geq \sigma. \label{MildSolLinHom}
\end{eqnarray}
In particular, we analyze the existence and uniqueness of $\mu$-pseudo almost periodic solutions of the following linear inhomogeneous equations:
\begin{eqnarray}
u^{\prime}(t)= A(t)u(t) + h(t) \quad \text{for all}\; t \in \mathbb{R}. \label{EqLinHomo}
\end{eqnarray}

 
\begin{theorem}\cite[Theorem 3.5]{AkdSou}
 Let $ \mu \in \mathcal{M}$ satisfy \textbf{(M)} and $ h \in PAPS^{p}(\mathbb{R},X, \mu) $. Assume that \textbf{(H1)}-\textbf{(H3)} hold. Then the abstract Cauchy problem \eqref{EqLinHomo} has a unique $\mu$-pseudo almost periodic mild solution $ u: \mathbb{R}\longrightarrow X ,$ given by
\begin{eqnarray}
u(t)=\int_{\mathbb{R}}G (t,s)h(s)\, ds  \quad \text{for all}\;  t\in  \mathbb{R}.\label{MildSolAp}
\end{eqnarray}
\end{theorem}

Now, we give our main result about the existence and uniqueness of $\mu$-pseudo almost periodic solutions to the equation \eqref{Eq_1}. We need the following Lemma:
\begin{lemma}\label{Lemma Bounded Sp}
Let for each $ x\in X $, $f(\cdot,x) \in PAPS^{p}(\mathbb{R},X,\mu )$ (for $ 1\leq p<\infty $). Assume that $ f $ satisfies hypothesis \textbf{(H4)}. Then the statement Theorem \ref{CompResMuChapter1}-\textbf{(iii)} holds for $f$.
\end{lemma}
\begin{proof}
Let $ B $ be any bounded set in $X$, i.e., there exists $ M \geq 0 $ such that $ \| x\| \leq M $ for all $x \in B$. Let $ x\in B $, by assumptions on $f$ we have
\begin{align*}
\int_{t}^{t+1} \|f(s,x) \|^{p} ds &\leq \int_{t}^{t+1} L_M (s)^{p} ds \| x\|^{p} +\int_{t}^{t+1} \|f(s,0) \|^{p} ds  \\
 &\leq M^{p}\| L_M\|^{p}_{BS^p} + \|f(\cdot,0) \|^{p}_{BS^p}, \quad t \in \mathbb{R}.   \\
\end{align*}
Thus, $$ \sup_{x \in B} \sup_{t \in \mathbb{R}} \left( \int_{t}^{t+1} \|f(s,x) \|^{p} ds\right)^{\frac{1}{p}} \leq  M\| L_M\|_{BS^p} + \|f(\cdot,0) \|_{BS^p} < \infty.$$
This proves the result.
\end{proof}
Now, let $u\in PAP(\mathbb{R},X,\mu)$. Then using Lemma \ref{Lemma Bounded Sp}, it follows in view of Theorem \ref{ThmComApSpChapter1} that $h:=f(\cdot,u(\cdot)) \in PAPS^{p}(\mathbb{R},X,\mu)$ for all $1\leq p<\infty$. Hence, the integral mapping defined by \eqref{MildSolAp} belongs to $ PAP(\mathbb{R},X,\mu )$ and  the map $F: PAP(\mathbb{R},X,\mu ) \longrightarrow PAP(\mathbb{R},X,\mu )$ given  by $$   Fu(t)=  \int_{\mathbb{R}}G (t,s)f(s,u(s))ds, \quad t\in \mathbb{R} $$ is well-defined. Moreover, we distinguish two cases: $p=1$ and $1<p<\infty$. For $1<p<\infty$, we have the following existence result.
\begin{theorem}\label{Main Theorem2 Eq1}
Let $ \mu \in \mathcal{M}$ satisfy \textbf{(M)} and $f(\cdot,x) \in PAPS^{p}(\mathbb{R},X,\mu )$ for each $x \in X$. Suppose that \textbf{(H1)}-\textbf{(H4)} hold and there exists $\rho >0$ such that 
\begin{eqnarray}
\rho > \left( 2  N  \left( \dfrac{2}{q\delta} \right)^{\frac{1}{q}}\dfrac{e^{\frac{\delta}{2}}}{e^{\frac{\delta}{2}}-1} \right) \|f(\cdot,0)\|_{BS^{p}} . \label{Contra cond 0 Theorem2}
\end{eqnarray}
\begin{eqnarray}
\| L_\rho \|_{BS^{p}} \leq  \left( 2  N  \left( \dfrac{2}{q\delta} \right)^{\frac{1}{q}}\dfrac{e^{\frac{\delta}{2}}}{e^{\frac{\delta}{2}}-1} \right)^{-1} -\rho^{-1}\|f(\cdot,0)\|_{BS^{p}} . \label{Contra cond Theorem2}
\end{eqnarray}
Then the equation \eqref{Eq_1} has a unique $\mu$-pseudo almost periodic solution $u$ with $$ \| u\|_{\infty} \leq \rho .$$
\end{theorem}
\begin{proof}
Define the mapping $F: \Lambda_{\rho}^{\mu}  \longrightarrow PAP(\mathbb{R},X,\mu )$  by 
$$   Fu(t)=  \int_{\mathbb{R}}G (t,s)f(s,u(s))ds, \quad t\in \mathbb{R}$$
where $\Lambda_{\rho}^{\mu} := \{ v \in PAP(\mathbb{R},X,\mu ) : \, \|v\|_{\infty} \leq \rho  \}.$ We show that $ F \Lambda_{\rho}^{\mu} \subset \Lambda_{\rho}^{\mu}$. Let $u \in \Lambda_{\rho}^{\mu}$. Then, by \textbf{(H4)}, we obtain that,
\begin{eqnarray*}
&&\|Fu(t)\| \\ 
&\leq & \int_{\mathbb{R}}\| G (t,s)f(s,u(s)) \| ds \\ 
&\leq &  N \int_{-\infty}^{t}  e^{- \delta (t-s)}  \| f(s,u(s))-f(s,0) \| ds +N \int_{-\infty}^{t}  e^{- \delta (t-s)}  \| f(s,0) \| ds  \\ &+& N \int^{+\infty}_{t} e^{- \delta (s-t)}  \|f(s,u(s))-f(s,0) \|  ds +N\int_{t}^{\infty}  e^{- \delta (t-s)}  \| f(s,0) \| ds \\
 &\leq & \rho N \int_{-\infty}^{t} e^{- \delta (t-s)} L_{\rho}(s) ds +N\int_{-\infty}^{t}  e^{- \delta (t-s)}  \| f(s,0) \| ds\\
 &+& \rho N \int^{+\infty}_{t} e^{- \delta (s-t)} L_{\rho}(s)   ds +N\int_{t}^{\infty}  e^{- \delta (t-s)}  \| f(s,0) \| ds\\
 &\leq &  N \left( \int_{-\infty}^{t} e^{- q\frac{ \delta}{2} (t-s)} ds\right)^{\frac{1}{q}} \left[\rho \left(  \int_{-\infty}^{t} e^{-p \frac{ \delta}{2} (t-s)} |L_{\rho}(s) |^{p} ds\right)^{\frac{1}{p}}+ \left(  \int_{-\infty}^{t} e^{-p \frac{ \delta}{2} (t-s)} \| f(s,0) \|^{p} ds\right)^{\frac{1}{p}}\right]  \\
 &+& N \left( \int^{+\infty}_{t} e^{- q\frac{ \delta}{2} (s-t)} ds\right)^{\frac{1}{q}}\left[  \rho  \left(  \int^{+\infty}_{t} e^{-p \frac{ \delta}{2} (s-t)} | L_{\rho}(s) |^{p} ds \right)^{\frac{1}{p}} +  \left(  \int^{+\infty}_{t} e^{-p \frac{ \delta}{2} (s-t)} \| f(s,0) \|^{p} ds \right)^{\frac{1}{p}} \right]  \\   
&\leq &  N \left( \int_{-\infty}^{t} e^{- q\frac{ \delta}{2} (t-s)} ds\right)^{\frac{1}{q}} \sum_{k\geq 1} e^{-k\frac{ \delta}{2} }\left[\rho \left(  \int_{t-k}^{t-k+1}  |L_{\rho}(s) |^{p} ds\right)^{\frac{1}{p}}+\left(  \int_{t-k}^{t-k+1} \| f(s,0) \|^{p} ds\right)^{\frac{1}{p}}\right]  \\
 &+& N \left( \int^{+\infty}_{t} e^{- q\frac{ \delta}{2} (s-t)} ds\right)^{\frac{1}{q}} \sum_{k\geq 1} e^{-k\frac{ \delta}{2} } \left[  \rho  \left(  \int^{t+k}_{t+k-1} | L_{\rho}(s) |^{p} ds \right)^{\frac{1}{p}} +  \left(  \int^{t+k}_{t+k-1}  \| f(s,0) \|^{p} ds \right)^{\frac{1}{p}} \right]  \\
& \leq & 2 N  \left( \dfrac{2}{q\delta} \right)^{\frac{1}{q}}\dfrac{e^{\frac{\delta}{2}}}{e^{\frac{\delta}{2}}-1}  \left(  \rho \| L_{\rho}\|_{BS^{p}} +\|f(\cdot,0)\|_{BS^{p}}  \right) 
\\ &\leq & \rho, \quad t\in \mathbb{R}.
\end{eqnarray*}
Hence, $ F \Lambda_{\rho}^{PAP} \subset \Lambda_{\rho}^{PAP}$. Let $u, \ v \in  \Lambda_{\rho}^{PAP}$. Then, a straightforward calculation yields
\begin{eqnarray*}
&&\|Fu(t)-Fv(t)\|\\
&\leq & \int_{\mathbb{R}}\| G (t,s)\left[ f(s,u(s))-f(s,u(s))\right]  \| ds \\ 
&\leq & N \int_{-\infty}^{t}  e^{- \delta (t-s)}  \| f(s,u(s))-f(s,v(s)) \| ds \\ &+&N  \int^{+\infty}_{t} e^{- \delta (s-t)}  \|f(s,u(s))-f(s,v(s)) \|  ds \\
&\leq &  N\left(  \int_{-\infty}^{t} e^{- \delta (t-s)} L_{\rho}(s) ds  + \int^{+\infty}_{t} e^{- \delta (s-t)} L_{\rho}(s)   ds \right)   \| u-v\|_{\infty} \end{eqnarray*}
\begin{eqnarray*}
&\leq &  N \sum_{k\geq 1} e^{-k\frac{ \delta}{2} } \left[\left( \int_{-\infty}^{t} e^{- q\frac{ \delta}{2} (t-s)} ds\right)^{\frac{1}{q}}  \left(  \int_{t-k}^{t-k+1}  L_{\rho}(s)^{p} ds  \right)^{\frac{1}{p}}\right. \\  
&+&\left.  \left( \int^{+\infty}_{t} e^{- q\frac{ \delta}{2} (s-t)} ds\right)^{\frac{1}{q}}\ \left( \int_{t+k}^{t+k-1}   L_{\rho}(s)^{p}   ds \right)^{\frac{1}{p}} \right]    \| u-v\|_{\infty} \\
&\leq & 2N   \left( \dfrac{2}{q\delta} \right)^{\frac{1}{q}}\dfrac{e^{\frac{\delta}{2}}}{e^{\frac{\delta}{2}}-1}  \| L_{\rho}\|_{BS^{p}} \| u-v\|_{\infty}, \; t\in \mathbb{R}. \\
\end{eqnarray*}
Thus, from \eqref{Contra cond 0 Theorem2} and \eqref{Contra cond Theorem2}, we obtain that
\begin{eqnarray*}
\| L_\rho \|_{BS^{p}} \leq \left( 2  N  \left( \dfrac{2}{q\delta} \right)^{\frac{1}{q}}\dfrac{e^{\frac{\delta}{2}}}{e^{\frac{\delta}{2}}-1} \right)^{-1} -\rho^{-1}\|f(\cdot,0)\|_{BS^{p}} <  \left( 2  N  \left( \dfrac{2}{q\delta} \right)^{\frac{1}{q}}\dfrac{e^{\frac{\delta}{2}}}{e^{\frac{\delta}{2}}-1} \right)^{-1}.
\end{eqnarray*}
Consequently, the mapping $F$ is a strict contraction in  $\Lambda_{\rho}^{PAP}$. Consequently, by  the Banach strict contraction principle, we obtain the existence and uniqueness of a solution $ u \in  \Lambda_{\rho}^{¨PAP}$. This proves the result.
\end{proof}
For $ p=1 $, we have:
\begin{theorem}\label{Main Theorem3 Eq1}
Let $ \mu \in \mathcal{M}$ satisfy \textbf{(M)} and $f(\cdot,x) \in PAPS^{1}(\mathbb{R},X,\mu )$ for each $x \in X$. Suppose that \textbf{(H1)}-\textbf{(H4)} hold and there exists $\rho >0$ such that 
\begin{eqnarray}
\rho > \left( 2  N  \dfrac{e^{\delta}}{e^{\delta}-1} \right) \|f(\cdot,0)\|_{BS^{1}} . \label{Contra cond 0 Theorem3}
\end{eqnarray}
\begin{eqnarray}
\| L_\rho \|_{BS^{1}} \leq \left(  2 N \dfrac{e^{\delta}}{e^{\delta}-1} \right)^{-1} - \rho^{-1}\|f(\cdot,0)\|_{BS^{1}}  . \label{Contra cond Theorem3}
\end{eqnarray}
Then the equation \eqref{Eq_1} has a unique $\mu$-pseudo almost periodic solution $u$ with $$ \| u\|_{\infty} \leq \rho .$$
\end{theorem}
\begin{proof}
We define the mapping $F: \Lambda_{\rho}^{\mu}  \longrightarrow PAP(\mathbb{R},X,\mu )$  by 
$$   Fu(t)=  \int_{\mathbb{R}}G (t,s)f(s,u(s))ds, \quad t\in \mathbb{R}.$$ 
Let $u \in \Lambda_{\rho}^{\mu}$. Arguing similarly as in the proof of Theorem \ref{Main Theorem3 Eq1} (without using H\"older inequality), we get that 
\begin{eqnarray*}
&&\|Fu(t)\| \\ 
&\leq & \int_{\mathbb{R}}\| G (t,s)f(s,u(s)) \| ds \\ 
&\leq &  N \int_{-\infty}^{t}  e^{- \delta (t-s)}  \| f(s,u(s))-f(s,0) \| ds +N \int_{-\infty}^{t}  e^{- \delta (t-s)}  \| f(s,0) \| ds  \\ &+& N \int^{+\infty}_{t} e^{- \delta (s-t)}  \|f(s,u(s))-f(s,0) \|  ds +N\int_{t}^{\infty}  e^{- \delta (t-s)}  \| f(s,0) \| ds \\
 &\leq & \rho N \int_{-\infty}^{t} e^{- \delta (t-s)} L_{\rho}(s) ds +N\int_{-\infty}^{t}  e^{- \delta (t-s)}  \| f(s,0) \| ds\\
 &+& \rho N \int^{+\infty}_{t} e^{- \delta (s-t)} L_{\rho}(s)   ds +N\int_{t}^{\infty}  e^{- \delta (t-s)}  \| f(s,0) \| ds
 \end{eqnarray*}
 \begin{eqnarray*}
& \leq & 2 N \dfrac{e^{\delta}}{e^{\delta}-1}  \left(  \rho \| L_{\rho}\|_{BS^{1}} +\|f(\cdot,0)\|_{BS^{1}}  \right) 
\\ &\leq & \rho, \quad t\in \mathbb{R}.
\end{eqnarray*}
Hence, $ F \Lambda_{\rho}^{PAP} \subset \Lambda_{\rho}^{PAP}$. Furthermore,  we have 
\begin{eqnarray*}
&&\|Fu(t)-Fv(t)\|\\
 &\leq & \int_{\mathbb{R}}\| G (t,s)\left[ f(s,u(s))-f(s,u(s))\right]  \| ds \\ 
&\leq & N \int_{-\infty}^{t}  e^{- \delta (t-s)}  \| f(s,u(s))-f(s,v(s)) \| ds \\ &+&N  \int^{+\infty}_{t} e^{- \delta (s-t)}  \|f(s,u(s))-f(s,v(s)) \|  ds \\
 &\leq &  N\left(  \int_{-\infty}^{t} e^{- \delta (t-s)} L_{\rho}(s) ds  + \int^{+\infty}_{t} e^{- \delta (s-t)} L_{\rho}(s)   ds \right)   \| u-v\|_{\infty} \\
&\leq & 2N   \dfrac{e^{\delta}}{e^{\delta}-1}  \| L_{\rho}\|_{BS^{1}} \| u-v\|_{\infty}, \; t\in \mathbb{R}. \\
\end{eqnarray*}
So, from \eqref{Contra cond 0 Theorem3} and \eqref{Contra cond Theorem3}, we obtain that
\begin{eqnarray*}
\| L_\rho \|_{BS^{1}} \leq  \left(2N \dfrac{e^{\delta}}{e^{\delta}-1}  \right)^{-1} -\rho^{-1}\|f(\cdot,0)\|_{BS^{1}} <  \left( 2N \dfrac{e^{\delta}}{e^{\delta}-1}  \right)^{-1}.
\end{eqnarray*}
Hence, the mapping $F$ is a strict contraction in  $\Lambda_{\rho}^{PAP}$ and then the result follows by the Banach strict contraction principle.
\end{proof}
\subsubsection{\textbf{Example}} 
Consider the following time--dependent parameters reaction--diffusion equation describing the behavior of bounded solutions to a one-species intraspecific competition Lotka--Volterra model, namely 
\begin{equation}
\left\{
    \begin{aligned}
  v_t (t,x)&=  \Delta v(t,x) -a(t)v(t,x) + b (t) v(t,x)^{2}+C(t,x), \  t \in \mathbb{R},\ x\in \Omega, \\   \label{model app2}
     v(t,x)&|_{\partial \Omega} = 0 ; \ t \in \mathbb{R},\ x\in \partial\Omega ,
    \end{aligned}
  \right. 
\end{equation}
where \\
$\bullet$ $\Omega \subseteq \R^N$ ($N \geq 1$) is an open bounded subset with a sufficiently smooth boundary.\\
$ \bullet $ $ \Delta $ is the Laplace operator on $\Omega  $, here the diffusion parameter equals $1$.\\
$ \bullet $ $ a \in AP(\mathbb{R}, [0,\infty))  $ with  $0<a_0:=\inf_{t \in \mathbb{R}} a(t) \leq a(t) \leq  \sup_{t \in \mathbb{R}} a(t)=a_1 <\infty$. It is assumed to be H\"older continuous with constant $L=1$ and exponent $\mu=1$.\\
$ \bullet $ The nonlinear term $ g:\mathbb{R}\times \mathbb{R} \longrightarrow  \mathbb{R} $ is defined by $$ g (t,v(t,x)) = b (t) v^{2}(t,x) +C(t,x), \; x \in \overline{\Omega}
$$
where $b  \in APS^{1}(\mathbb{R}, [0,\infty)) $.\\
$ \bullet $ $ C: \mathbb{R}\times \overline{\Omega} \longrightarrow (0,\infty)$ is locally integrable with respect to $t$ and continuous with respect to $x$.\\

In order to transform our model \eqref{model app2} to the abstract form, i.e., the equation \eqref{Eq_1}, we define the Banach space $X=C(\overline{\Omega})$ equipped with the sup-norm. Set the linear operators $ (A(t) , D(A(t) )) $, $ t\in \mathbb{R} $ by 
\begin{equation}
 \left\{
    \begin{aligned}
     A(t) \varphi &:= \Delta  +a(t) ,\\   
     D(A(t) )& = \left\{  \varphi \in C(\overline{\Omega}) \cap H_{0}^{1}\left(  \Omega \right)  :\  \Delta \varphi  \in C(\overline{\Omega})  \right\}:=D.
    \end{aligned} \label{EqApp1}
  \right. 
\end{equation}
Here $ \Delta $ is the Laplacian in the sense of distributions on $ \Omega $. It is well-known (see \cite[Proposition 14.6]{DaSi}) that $ (\Delta, D)  $ generates a contraction analytic semigroup $(e^{s \Delta})_{s \geq 0}$ on $X$, with angle $ \phi \in (\frac{\pi}{2}, \pi) $ such that 
\begin{eqnarray}
\| e^{s \Delta} \| \leq 1, \; s\geq 0
\end{eqnarray}
and 
\begin{eqnarray}
\| R(\lambda, \Delta) \| \leq 1/\lambda, \quad \lambda > 0.
\end{eqnarray}
Moreover, the semigroup is exponentially stable and we have
\begin{eqnarray}
\| e^{s \Delta} \| \leq e^{-\lambda_{1}s}, \; s\geq 0,
\end{eqnarray}
  where $\lambda_1$ is the smallest eigenvalue of $-\Delta$. Therefore, bu assuming that $ a_0 <0 $ we obtain that for each $ t\in \mathbb{R} $, $(A(t),D)$ generates an analytic semigroup $ (T^{t}(s))_{s \geq 0} $ of type $ (1,\omega := a_0 +\lambda_1 >1) $ on  $X$. Therefore each semigroup  $ (T^{t}(s))_{s \geq 0} $ is exponentially stable. Moreover, we have
$$ \sup_{t,s \in \mathbb{R}}\|A(t)A(s)^{-1} \| \leq 1+ 2a_1 \| A(s)^{-1}\|\leq 1+\dfrac{2a_1 }{a_0+ \lambda_1}, \, t,s \in \mathbb{R} .  $$
Thus, $$ \sup_{t,s \in \mathbb{R}}\|A(t)A(s)^{-1} \| <\infty .  $$
Furthermore, by our assumption on $a$ we obtain that 
\begin{eqnarray}
\| (A(t)-A(s))A(r)^{-1} \| = |a(t)-a(s) |\| A(r)^{-1}\|  \leq \omega^{-1}  |t-s |  \quad \text{for } t,s,r \in \mathbb{R}.
\end{eqnarray}
Since $\omega^{-1} \int_{0}^{\infty} e^{-\omega \tau} d \tau =\omega ^{-2} <1$ we obtain from \textbf{(C2)} that $ (A(t),D)_{t\in \mathbb{R}} $ generates an exponentially stable analytic evolution family with exponent $0<\delta< \dfrac{\omega ^2 -1}{2 \omega }$. Furthermore, the fact that $ a \in  AP(\mathbb{R}) $ implies $ A(\cdot)^{-1}  \in AP(\mathbb{R},\mathcal{L}(X))$. Indeed, let $t, \tau \in \mathbb{R}$. So, we have
\begin{eqnarray*}
A(t +\tau )^{-1}\varphi- A(t)^{-1}\varphi &=& A(t +\tau )^{-1}(A(t +\tau )-A(t) )A(t)^{-1} \varphi \nonumber \\
&=& (a(t+\tau)-a(t))A(t +\tau )^{-1}A(t)^{-1}, \; \varphi \in X.
\label{AP Green function form1}
\end{eqnarray*}
Thus, the result follows from the next inequality:
\begin{eqnarray*}
\| A(t +\tau )^{-1}\varphi - A(t)^{-1}\varphi  \|& \leq &| a(t+\tau)-a(t) |\| A(t +\tau )^{-1}\| _{\mathcal{L}(X)}\| A(t)^{-1} \varphi \|_{D} \\
& \leq & C | a(t+\tau)-a(t) | \| \varphi\|, \; \varphi \in X.
\end{eqnarray*}
Therefore, \textbf{(H3)} is also satisfied. In order to check \textbf{(H4)}, we define the superposition operator $ f: \mathbb{R}\times X\longrightarrow X $ by $$ f(t,\varphi)(x):=g(t,\varphi(x)), \; x\in \overline{\Omega} .$$
Hence, we obtain the following result
\begin{proposition}
The function $f$ satisfies hypothesis \textbf{(H4)}.
\end{proposition}
\begin{proof}
Let $ \varepsilon, \psi \in X $ and let $ \rho >0 $ be such that $ \|\varphi \|, \|\psi \| \leq \rho $. Then, we have
\begin{eqnarray*}
| g(t,\varphi(x))-g(t,\psi(x))|&=&| b(t) | |   \varphi(x)-\psi(x)  || \varphi(x)+\psi(x)  | \\ &\leq &
| b(t) | (\|\varphi \|+\| \psi\|) \|   \varphi-\psi \| \\ &\leq &
2 \rho| b(t) | \|   \varphi-\psi \|, \; t \in \mathbb{R}, \, x\in  \overline{\Omega}.
\end{eqnarray*}
Hence, 
\begin{eqnarray*}
\| f(t,\varphi)-f(t,\psi)\|\leq 
2 \rho| b(t) | \|   \varphi-\psi \|.
\end{eqnarray*}
Since $ b \in BS^{1}(\mathbb{R},[0, \infty))  $, it follows that \textbf{(H4)} holds with $ L_\rho=2 \rho b(\cdot)  $.
\end{proof}

As a consequence of Theorem \ref{Main Theorem3 Eq1}, we obtain the following main result.

\begin{theorem}
The equation \eqref{EqApp1} has a unique $\mu $-pseudo almost periodic solution satisfying $\sup_{(t,x) \in \mathbb{R}\times \overline{\Omega}}| v(t,x)|\leq \rho $ for some $\rho >\dfrac{2e^{\delta}}{e^{\delta}-1} \|C \|_{BS^{1}}$ provided that, $$ | b|_{BS^{1}} \leq  \dfrac{e^{\delta}-1}{2 e^{\delta}}-\dfrac{\|C \|_{BS^{1}}}{\rho }.$$ 
\end{theorem}

\end{document}